\documentclass[reqno,centertags,12pt]{amsart}
\usepackage[letterpaper,margin=1.3in]{geometry}
\usepackage{amsmath,amsthm,amsfonts,amssymb,enumerate}
\usepackage[bookmarksopen=false,final,hidelinks]{hyperref}

\newcommand{\bb}{\mathbb}
\newcommand{\ol}{\overline}
\newcommand{\ul}{\underline}
\newcommand{\bs}{\backslash}
\newcommand{\supp}{\mathrm{supp}}
\newcommand{\cvh}{\mathrm{cvh}}
\newcommand{\PW}{\mathcal{PW}}
\newcommand{\BC}{\mathcal{BC}}
\newcommand{\ca}{\mathrm{Cap}}
\newcommand{\dist}{\operatorname{dist}}
\newcommand{\eps}{\varepsilon}

\newtheorem{theorem}{Theorem}
\newtheorem{proposition}[theorem]{Proposition}
\newtheorem{lemma}[theorem]{Lemma}
\newtheorem{corollary}[theorem]{Corollary}
\theoremstyle{definition}
\newtheorem{example}[theorem]{Example}
\numberwithin{equation}{section}
\numberwithin{theorem}{section}

\begin{document}

\title{A four-way Szeg\H{o} theorem for $L^p$ extremal polynomials on subsets of $\bb R$}
	
\author{G\"{o}kalp Alpan}
\address{Faculty of Engineering and Natural Sciences, Sabanc\i\ Univ., \.Istanbul, T\"urkiye}
\email{gokalp.alpan@sabanciuniv.edu}
\thanks{\footnotesize G.A. and M.Z. are supported by the Scientific and Technological Research Council of T\"urkiye (T\"UB\.ITAK) ARDEB 1001 Grant Number 123F358.}
	
\author{Maxim Zinchenko}
\address{Department of Mathematics and Statistics\\ University of New Mexico\\ Albuquerque, NM 87131, USA}
\email{maxim@math.unm.edu}
\thanks{\footnotesize M.Z. is supported in part by the Simons Foundation Grant MP-TSM-00002651.}
	
\subjclass[2020]{Primary 41A17; Secondary 41A50, 42C05, 41A44}
\keywords{Widom factors, Chebyshev polynomials, orthogonal polynomials, Parreau--Widom sets, Blaschke condition, Szeg\H{o} condition, Szeg\H{o} theorem.}
	
\begin{abstract}
We prove a four-way Szeg\H{o} theorem for $L^p$ extremal polynomials on
compact supports $K=K_0\cup X\subset\bb R$, where $K_0$ is a regular compact set and $X$ is a finite or countable set of isolated points.  For every $2\le p\le\infty$, including the weighted Chebyshev case $p=\infty$ under the corresponding assumptions on the weight, any three of the
Parreau--Widom condition for $K_0$, the Blaschke condition for $X$, the
Szeg\H{o} condition for the weight, and the Widom condition $0<\limsup_{n\rightarrow\infty}W_{p,n}<\infty$ imply the fourth.
As a consequence, for every regular compact set $K\subset\bb R$ of positive capacity, the Parreau--Widom condition is equivalent both to boundedness of the unweighted Chebyshev Widom factors and to boundedness of the
equilibrium-measure $L^2$ Widom factors.
For every $0<p\le\infty$, we also prove upper and lower bounds for the Widom factors in which the contributions of the weight, the isolated points, and the gaps of $K_0$ appear separately.
Finally, we give examples illustrating the sharpness of our results.  For $p=2$, we realize every combination of the following five properties that is not excluded by the implications proved in this work: the Parreau--Widom condition, the Blaschke condition, the Szeg\H{o} condition, boundedness of the Widom factors from above, and boundedness of the Widom factors away from zero.
\end{abstract}
	
\maketitle

\section{Introduction}

Let $K\subset\bb R$ be a non-polar compact set.  We denote its logarithmic
capacity by $\ca(K)$ and its equilibrium measure by $\mu_K$.  The Green
function of $\ol{\bb C}\bs K$ with pole at infinity is
\begin{equation}\label{GrFn}
    g_K(z)=-\log\ca(K)+\int_K\log|z-\zeta|\,d\mu_K(\zeta).
\end{equation}
Thus $g_K(z)=\log|z|-\log\ca(K)+o(1)$ as $z\rightarrow\infty$.  For
$x\in K$, we set
$g_K(x)=\limsup_{\bb C\bs K\ni z\rightarrow x}g_K(z)$.  A point $x$ is
regular when $g_K(x)=0$, and $K$ is regular when all of its points are
regular.  If $K_{\rm reg}$ and $K_{\rm ir}$ denote the regular and irregular
points, respectively, then $K$ is called semi-regular when $K_{\rm reg}$ is
closed.  Kellogg's theorem says that $K_{\rm ir}$ is polar
\cite[Theorem~4.2.5]{Ran95}.  In the semi-regular case, $K_{\rm reg}$ is a
regular compact set and
$\ca(K)=\ca(K_{\rm reg})$, $g_K=g_{K_{\rm reg}}$
\cite[Section~2.1, p.~5]{DukZin26}.  The equality
$\mu_K=\mu_{K_{\rm reg}}$ then follows from \eqref{GrFn}.

Let $\mu$ be a finite positive Borel measure with non-polar compact support
$K=\supp(d\mu)\subset\bb R$.  In particular, $K$ is infinite.  For
$0<p<\infty$, put
\begin{equation*}
    t_{p,n}(K,\mu)=
    \inf_{\substack{P\ {\rm monic}\\ \deg P=n}}
    \left(\int_K |P|^p\,d\mu\right)^{1/p},
    \qquad
    W_{p,n}(K,\mu)=\frac{t_{p,n}(K,\mu)}{\ca(K)^n}.
\end{equation*}
Similarly, let $K\subset\bb R$ be a non-polar compact set and let
$u:K\rightarrow[0,\infty)$ be a bounded weight satisfying
$\supp(u):=\ol{\{x\in K:u(x)>0\}}=K$.  We write $\|f\|_K=\sup_{x\in K}|f(x)|$ and set
\begin{equation*}
    t_n(K,u)=
    \inf_{\substack{P\ {\rm monic}\\ \deg P=n}}\|uP\|_K,
    \qquad
    W_{\infty,n}(K,u)=\frac{t_n(K,u)}{\ca(K)^n}.
\end{equation*}
For $p=2$, the extremal polynomials are the monic orthogonal polynomials.  For $p=\infty$, they are the weighted Chebyshev polynomials.  The quantities
$W_{p,n}$ are called the Widom factors.  The capacity normalization removes the common exponential growth, since the Fekete--Szeg\H{o} theorem gives
$t_n(K,1)^{1/n}\rightarrow\ca(K)$ \cite[Corollary~5.5.5]{Ran95}.
The question is then whether the normalized norms stay bounded above, stay away from zero, or have a more singular asymptotic behavior.

For $u\ge0$ in $L^1(\mu_K)$, define
$S(K,u)=\exp(\int_K\log u\,d\mu_K)$, with $S(K,u)=0$ when the integral is
$-\infty$.  The condition $S(K,u)>0$ is the Szeg\H{o} condition.  It measures
the size of the density relative to equilibrium measure, which is the natural
reference measure for the capacity normalization.  On one interval, after an
affine change of variables, if $d\mu=w\,d\mu_K+d\mu_s$,
$d\mu_s\perp d\mu_K$, and $S(K,w)>0$, then Szeg\H{o}'s theorem gives
$\lim_{n\rightarrow\infty}W_{2,n}(K,\mu)^2=2S(K,w)$
\cite[Theorem~13.8.8]{Sim05}.  Peherstorfer and Yuditskii additionally allowed
a denumerable set $Y$ of exterior mass points satisfying
$\sum_{y\in Y}g_K(y)<\infty$ \cite[Theorem, p.~3215]{PehYud01}.  For such a
measure, supported on $K\cup Y$, their formula contains the additional factor
$\exp(\sum_{y\in Y}g_K(y))$ in the limit of $W_{2,n}(K\cup Y,\mu)$.

Strong asymptotics for weighted $L^p$ extremal polynomials on $[-1,1]$,
$1<p<\infty$, were obtained under assumptions on the weight stronger than the
Szeg\H{o} condition \cite[Theorem~1.1]{LS87}.  For $0<p<\infty$, Kaliaguine
proved strong asymptotics for Szeg\H{o}-class measures on sufficiently smooth
closed curves, allowing a singular part on the curve and finitely many exterior mass points \cite[Theorem~2.2]{Kal93}.
For $p=2$, the analogous result for absolutely continuous
Szeg\H{o}-class measures on $C^{2+}$ arcs with finitely many exterior
mass points was proved in \cite[Theorem~1]{Kal95}.

Throughout the paper, unless explicitly stated otherwise, we assume that
\begin{equation}\label{support-main-assumptions}
\begin{aligned}
&K=K_0\cup X\subset\bb R\text{ is compact},\\
&K_0\text{ is a regular compact set with }\ca(K_0)>0,\\
&X\subset\bb R\bs K_0\text{ is at most countable and every point of }X \text{ is isolated in }K.
\end{aligned}
\end{equation}
Fix an enumeration $X=\{x_j\}$ without repetition when
$X\ne\varnothing$.  Whenever a sum involves the points $x_j$ or the corresponding masses $a_j$, it is taken over this enumeration and is understood to be zero when $X=\varnothing$.

Since $X$ is at most countable, it is polar
\cite[Corollary~3.2.5]{Ran95}.  Thus $K$ is semi-regular,
$K_{\rm reg}=K_0$, $K_{\rm ir}=X$, and
$\ca(K)=\ca(K_0)$, $g_K=g_{K_0}$
\cite[Section~2.1, p.~5]{DukZin26}.  The equality $\mu_K=\mu_{K_0}$ follows from \eqref{GrFn}, and
$g_K(x_j)>0$ for every $x_j\in X$.

Let $(\alpha_\ell,\beta_\ell)$ be the bounded gaps of $K_0$.  On each gap,
\eqref{GrFn} gives
$g_K''(x)=-\int_K(x-\zeta)^{-2}\,d\mu_K(\zeta)<0$, so $g_K$ is strictly
concave.  Since $K_0$ is regular, $g_K$ tends to zero at both endpoints and is
positive in the gap.  It therefore has a unique critical point $c_\ell$, where
it reaches its maximum.  See also \cite[Section~2.1]{Chr12}.  We use the two sums
\begin{equation}\label{PW-BC-def}
    \PW(K)=\sum_\ell g_K(c_\ell)=\PW(K_0),
    \qquad
    \BC(K)=\sum_{j} g_K(x_j).
\end{equation}
The regular part $K_0$ is Parreau--Widom exactly when $\PW(K)<\infty$, and the
isolated points satisfy the Blaschke condition exactly when $\BC(K)<\infty$.
In \eqref{PW-BC-def}, $\PW(K)$ records the total Green height of the gaps, while
$\BC(K)$ records the total Green size of the isolated points.  In the terminology of
\cite[(2.4), p.~5]{DukZin26}, $\BC(K)$ is the irregularity coefficient of $K$ with
pole at infinity.

Assuming \eqref{support-main-assumptions}, for $0<p<\infty$ we consider
finite positive Borel measures with $\supp(d\mu)=K$ of the form
\begin{equation}\label{mu-main-decomp}
    d\mu=w\,d\mu_K+d\sigma+\sum_{j}a_j\,d\delta_{x_j},
    \qquad
    d\sigma\perp d\mu_K,\quad
    a_j>0,\qquad
    \sum_j a_j<\infty,
\end{equation}
where $d\sigma$ is supported on $K_0$ and is otherwise arbitrary.  Let
$\mu_{\rm ac}$ be the part of $\mu|_{K_0}$ absolutely continuous with respect
to $\mu_K$, and fix a nonnegative Borel representative
$w=\frac{d\mu_{\rm ac}}{d\mu_K}$.  Thus, for finite $p$, the density $w$ is taken with respect to equilibrium measure, not Lebesgue measure.  For $p=\infty$, the symbol $w$ instead denotes a given bounded weight on $K$ with $\supp(w)=K$.  In both cases, $S(K,w)=\exp(\int_{K_0}\log w\,d\mu_K)$.  For notational uniformity, when $p=\infty$ we write
\begin{equation*}
    W_{\infty,n}(K,\mu):=W_{\infty,n}(K,w),
\end{equation*}
where the notation on the left is shorthand and does not involve a measure.
Finally, let $S_p(K,w)=S(K,w)^{1/p}$ for $0<p<\infty$ and
$S_p(K,w)=S(K,w)$ for $p=\infty$.

Put
\begin{equation*}
    \ul W_p(K,\mu)=\liminf_{n\rightarrow\infty}W_{p,n}(K,\mu),
    \qquad
    \ol W_p(K,\mu)=\limsup_{n\rightarrow\infty}W_{p,n}(K,\mu).
\end{equation*}
We use the following conditions.  Their relations are studied systematically
in Section~\ref{sec-szego-theorems}:
\begin{equation}\label{intro-four-properties}
\begin{array}{cl@{\qquad}cl}
    (P):&K_0\text{ is Parreau--Widom},
    &(B):&\BC(K)<\infty,\\
    (S):&S(K,w)>0,
    &(U_p):&\ol W_p(K,\mu)<\infty,\\
    (L_p):&\ul W_p(K,\mu)>0,
    &(W_p):&0<\ol W_p(K,\mu)<\infty.
\end{array}
\end{equation}
Thus $(U_p)+(L_p)$ implies $(W_p)$, whereas $(W_p)$ does not in general imply
$(L_p)$.

For the unweighted Chebyshev problem on a finite-gap set $K\subset\bb R$, the
factors $W_{\infty,n}(K,1)$ are asymptotic to an almost periodic sequence
\cite[Theorem~1.8 and Remark~3 following Theorem~1.6]{CSZ17I}, and need not
converge even when $K$ consists of two intervals \cite[p.~128]{Wid69}.  The
unweighted Chebyshev polynomials on every finite-gap set also have
Szeg\H{o}--Widom asymptotics \cite[Theorem~1.9]{CSZ17I}.  Beyond the finite-gap
setting, strong Szeg\H{o}--Widom asymptotics for the unweighted Chebyshev
polynomials on regular Parreau--Widom sets satisfying the direct Cauchy theorem
condition were proved in \cite[Theorem~1.3]{CSYZ19}.

For measures of the form above whose regular part $K_0$ is finite-gap, any two
among $(S)$, $(B)$, and the two-sided $L^2$ Widom condition
$(U_2)+(L_2)$ imply the third \cite[Theorem~4.1]{CSZ11}.  When $K_0$ is a
regular Parreau--Widom set, Christiansen proved that, under $(B)$, condition
$(S)$ is equivalent to $(W_2)$, and either condition yields $(U_2)$ and $(L_2)$
\cite[Theorem~2 and the following corollary]{Chr12}.  His result already allows a finite or denumerable exterior point spectrum.
Strong asymptotics for orthogonal polynomials on homogeneous sets with
exterior masses, under an additional condition on the poles and zeros of
the associated Stieltjes function, were proved in
\cite[Theorem in Section~6, pp.~144--145]{PehYud03}.

The quantitative estimates closest to ours come from several parts of the
literature.  For $0<p<\infty$ and without isolated points,
\cite[Theorem~2.1]{AZ20a} gives the fixed-degree bound
$W_{p,n}(K,\nu)\ge S(K,\rho)^{1/p}$ for
$d\nu=\rho\,d\mu_K+d\nu_s$, $d\nu_s\perp d\mu_K$.  Its sharpness was proved
in \cite[Theorem~2.2]{AZ20a}, and sharpness on a fixed real compact set was
proved in \cite[Theorem~2.2]{AZ20b}.  On regular Parreau--Widom sets, the
weighted Chebyshev upper bound and the corresponding Szeg\H{o} theorem were
proved in \cite[Theorems~3.2 and 3.3]{CSZ26}.  For the unweighted problem on the
semi-regular supports considered here, Dukes and Zinchenko proved
$W_{\infty,n}(K,1)\le 2e^{\PW(K)+\BC(K)}$ for every $n\ge1$
\cite[Theorem~3.3, (3.7), p.~13]{DukZin26} and
$\liminf_{n\rightarrow\infty}W_{\infty,n}(K,1)\ge2e^{\BC(K)}$
\cite[Corollary~4.5, (4.6), p.~17]{DukZin26}.  Under $(P)$, uniform boundedness of
$\{W_{\infty,n}(K,1)\}$ is equivalent to $(B)$
\cite[Theorem~5.1, p.~17]{DukZin26}.  When $K_0$ is an interval,
$W_{\infty,n}(K,1)\rightarrow2e^{\BC(K)}$
\cite[Theorem~5.4, (5.3), p.~18]{DukZin26}.

Theorem~\ref{main-estimates-thm} gives the following estimates for every
$0<p\le\infty$ and for countably many isolated points.  At finite $p$, the
singular part remains arbitrary.  For $p=\infty$, the weight is assumed to be
bounded on $K$ and to satisfy $\supp(w)=K$.  The upper estimate also requires
the weight to be upper semicontinuous.
\begin{align}
    \liminf_{n\rightarrow\infty}W_{p,n}(K,\mu)
    &\ge S_p(K,w)e^{\BC(K)}
    &&\text{if }(S)\text{ holds},
    \label{intro-main-lower}\\
    \limsup_{n\rightarrow\infty}W_{p,n}(K,\mu)
    &\le 2e^{\PW(K)+\BC(K)}S_p(K,w)
    &&\text{if }(P)\text{ and }(B)\text{ hold}.
    \label{intro-main-upper}
\end{align}
Here $S_p(K,w)$ records the contribution of the density/weight $w$, $e^{\BC(K)}$ records the contribution of the isolated points, and
$e^{\PW(K)}$ records the contribution of the gaps of the regular part.

Unlike the upper estimate \eqref{intro-main-upper}, the lower estimate
\eqref{intro-main-lower} needs no Parreau--Widom assumption.  If $(S)$
holds but $(B)$ fails, its right-hand side is $+\infty$, and hence the
Widom factors tend to infinity.  The remaining case
$S(K,w)=0$ and $\BC(K)=\infty$ lies outside the hypotheses of
Theorem~\ref{main-estimates-thm}, since the product
$S_p(K,w)e^{\BC(K)}$ then has the indeterminate form
$0\cdot\infty$.
This indeterminacy is genuine.  Already for $p=2$, the Widom factors can
tend to zero, remain identically one, tend to infinity, or have lower
limit zero and upper limit infinity.  These cases occur, respectively,
in the constructions for Corollary~\ref{discrete-patterns-cor}(a),
Example~\ref{jacobi-patterns-ex}(a) and (b), and Corollary~\ref{discrete-patterns-cor}(d).  Consequently, no single value assigned to $0\cdot\infty$ can make both \eqref{intro-main-lower} and \eqref{intro-main-upper} valid in general.

Theorem~\ref{szego-implications-thm} gives the direct and converse
implications of \eqref{intro-main-lower} and \eqref{intro-main-upper}.  The
lower estimate gives
$(S)\Rightarrow(L_p)$ and, if $(B)$ fails, $W_{p,n}(K,\mu)\rightarrow\infty$.
Consequently, $(S)+(U_p)\Rightarrow(B)$.  The upper estimate gives
$(P)+(B)\Rightarrow(U_p)$ and, if $(S)$ fails, $W_{p,n}(K,\mu)\rightarrow0$.
Consequently, $(P)+(B)+(L_p)\Rightarrow(S)$.
For $p=\infty$, the lower conclusions require the weight to be bounded on
$K$ and to satisfy $\supp(w)=K$, while the upper conclusions additionally
require the weight to be upper semicontinuous.

Under $(P)$, Corollary~\ref{3szego-thm} gives the following conclusions for
every $0<p\le\infty$: $(B)+(S)$ implies both $(U_p)$ and $(L_p)$,
$(B)+\neg(S)$ gives convergence to zero, and $(S)+\neg(B)$ gives divergence
to infinity.

For $2\le p\le\infty$, these implications close into a four-way Szeg\H{o}
theorem.  Corollary~\ref{four-way-thm} states that any three of $(P)$, $(B)$,
$(S)$, and $(W_p)$ imply the fourth.  The additional implication is supplied
by Theorem~\ref{unconditional-pw-thm}, which gives
$(S)+(U_p)\Rightarrow(P)+(B)$ and the quantitative estimate
\begin{equation*}
    \PW(K)
    \le 2\log\frac{\ol W_p(K,\mu)}{S_p(K,w)}-\log\frac4\pi.
\end{equation*}

As a consequence of the four-way Szeg\H{o} theorem,
Corollary~\ref{regular-pw-characterization-cor} gives, for every regular
compact set $K\subset\bb R$ of positive capacity,
\begin{equation*}
\begin{split}
    K\text{ is Parreau--Widom}
    &\Longleftrightarrow
    \sup_{n\ge1}W_{\infty,n}(K,1)<\infty\\
    &\Longleftrightarrow
    \sup_{n\ge1}W_{2,n}(K,\mu_K)<\infty.
\end{split}
\end{equation*}

This settles several questions in the literature and strengthens earlier partial converse results.
Christiansen, Simon, and Zinchenko asked whether the middle-thirds Cantor set, which is not Parreau--Widom, could nevertheless have bounded Chebyshev Widom factors \cite[Remark~3 following Theorem~1.4]{CSZ17I}.  Alpan and Goncharov asked more generally whether a regular non-Parreau--Widom compact subset of $\bb R$ can have bounded Chebyshev Widom factors, whether zero Lebesgue measure forces those factors to be unbounded, and the corresponding questions for the $L^2$ Widom factors of the equilibrium measure \cite[Problem~8]{AlpGon17Q}.  The Chebyshev question was posed again in \cite[Open Problem~2.2]{CSZ22}.  The characterization above answers these questions within the class of regular compact subsets of $\bb R$.

There were also partial converse results under additional hypotheses. Under the assumption that $K$ has a canonical generator, \cite[Theorem~1.4]{CSYZ19} proved that bounded Chebyshev Widom factors imply the Parreau--Widom condition.  The analogous implication for the $L^2$ Widom factors and the equivalence of the two boundedness conditions, under the corresponding character-density hypothesis, were obtained in \cite[Theorem~1.4 and Corollary~1.5]{Alp19}.  The characterization above removes these additional hypotheses for regular compact subsets of $\bb R$.

Section~4 examines the sharpness of the preceding implications.  Several
of its constructions are valid over broad ranges of $p$ and show that
natural converses and stronger conclusions fail without additional
assumptions.  For $p=2$, the analysis is complete: every sign pattern not
excluded by Theorems~\ref{szego-implications-thm} and
\ref{unconditional-pw-thm} is realized.

The elementary and discrete constructions account for eleven of the
seventeen sign patterns not excluded for $p=2$.  The remaining six are
realized by Examples~\ref{ex-B-L2-only}, \ref{ex-only-L2},
\ref{jacobi-patterns-ex}(a)--(b), and
\ref{zero-measure-fixed-ex}(a)--(b).  Consequently, for $p=2$, no
additional implication among $(P)$, $(B)$, $(S)$, $(U_2)$, and $(L_2)$
can hold in general.

The rest of the paper is organized as follows.  Section~2 establishes the
upper and lower bounds for $L^p$ Widom factors used in the sequel.
Section~3 proves the Szeg\H{o}-type implications and equivalences stated
above.  Section~4 develops the examples and the sign-pattern
classification summarized above.

\section{Upper and Lower Bounds for \texorpdfstring{$L^p$}{Lp} Widom Factors}

All inequalities below are understood in the extended-real sense.  In particular, $e^{+\infty}=+\infty$ and
$s\cdot(+\infty)=+\infty$ for $s>0$.  The hypotheses are chosen so that products of Szeg\H{o} factors with exponentials of Green sums are never of the indeterminate form $0\cdot\infty$.

The case $u\equiv1$ of the estimate below follows from the
\emph{Totik--Widom upper bound} of Dukes and Zinchenko
\cite[Theorem~3.1, (3.2), p.~11]{DukZin26}.  For general weights, we use an outer
finite-gap approximation and the non-asymptotic reciprocal-polynomial
estimate \cite[Theorem~2.8]{CSZ26}.

\begin{proposition}\label{CSZ-upper-input}
Let $K\subset\bb R$ be a non-polar semi-regular compact set, let $g_K$ be its Green function, and put $H(K)=\sum_G\sup_{x\in G}g_K(x)$, where $G$ ranges over the bounded components of $\cvh(K)\bs K$.  If $H(K)<\infty$ and $u:K\rightarrow[0,\infty)$ is bounded, upper semicontinuous, and satisfies $\supp(u)=K$, then
\begin{equation}\label{CSZ-upper-input-eq}
    \limsup_{n\rightarrow\infty}W_{\infty,n}(K,u)
    \le
    2S(K,u)\exp[H(K)].
\end{equation}
\end{proposition}

\begin{proof}
We first consider a reciprocal-polynomial weight $v=1/|Q_m|$, where $Q_m$ is
a real polynomial of degree $m$ with no zeros on $K$.  For $\delta>0$, put
$K_\delta=\{x\in\cvh(K):\dist(x,K)\le\delta\}$.  This set is a finite union of
intervals: its bounded gaps are the intervals
$(a+\delta,b-\delta)$ for those bounded gaps $(a,b)$ of $K$ with
$b-a>2\delta$, and there are only finitely many of them.  In particular,
$K_\delta$ is regular and Parreau--Widom.

Since $K\subset K_\delta$, monotonicity of Green functions gives
$g_{K_\delta}\le g_K$ on $\bb C\bs K_\delta$.  Different bounded gaps of
$K_\delta$ are contained in different bounded gaps of $K$, and hence
\begin{equation}\label{outer-PW-bound}
    \PW(K_\delta)
    \le
    \sum_G\sup_{x\in G}g_K(x)
    =
    H(K).
\end{equation}
As $\delta\downarrow0$, the sets $K_\delta$ decrease to $K$.  Continuity from
above of capacity and uniqueness of the equilibrium measure give
$\ca(K_\delta)\rightarrow\ca(K)$ and
$\mu_{K_\delta}\rightarrow\mu_K$ weakly
\cite[Theorem~5.1.3(a) and its proof]{Ran95}.  For all sufficiently small
$\delta$, the polynomial $Q_m$ has no zeros on $K_\delta$, and $\log v$ is
continuous on one fixed neighborhood of $K$.  It follows that
$S(K_\delta,v)\rightarrow S(K,v)$.

Fix $n>m$.  By \eqref{outer-PW-bound} and
\cite[Theorem~2.8]{CSZ26}, applied to $K_\delta$ with pole at infinity,
\begin{align*}
    t_n(K,v)
    &\le
    t_n(K_\delta,v)\\
    &<
    2\ca(K_\delta)^nS(K_\delta,v)\exp[\PW(K_\delta)]\\
    &\le
    2\ca(K_\delta)^nS(K_\delta,v)\exp[H(K)].
\end{align*}
Letting $\delta\downarrow0$ gives
$W_{\infty,n}(K,v)\le2S(K,v)\exp[H(K)]$ for every $n>m$.

Now let $u$ be as in the statement.  Fix $\eps>0$, set $C=\|u\|_K$, and let
$f_\eps=1/(u+\eps)$, which is positive and lower semicontinuous.  For
$k\ge1$, put
$q_k(x)=\inf_{y\in K}\{f_\eps(y)+k|x-y|\}$.  Then $q_k$ is continuous,
$q_k\uparrow f_\eps$ pointwise on $K$, and $q_k\ge1/(C+\eps)$.  Uniform
polynomial approximation gives real polynomials $r_k$ such that
$\|r_k-q_k\|_K<\eta_k$, where $\eta_k\downarrow0$ and
$\eta_k<1/[4(C+\eps)]$.  Setting $P_k=r_k-\eta_k$, we have
$1/[2(C+\eps)]\le P_k\le f_\eps$ on $K$ and
$P_k\rightarrow f_\eps$ pointwise.

The weights $v_k=1/P_k=1/|P_k|$ satisfy
$u\le u+\eps\le v_k\le2(C+\eps)$ and converge pointwise to $u+\eps$.
Set $L=\limsup_{n\rightarrow\infty}W_{\infty,n}(K,u)$.  For each fixed $k$,
the reciprocal-polynomial case, valid once $n>\deg P_k$, and monotonicity in
the weight give $L\le2S(K,v_k)\exp[H(K)]$.  Since
$\eps\le v_k\le2(C+\eps)$, dominated convergence gives
$S(K,v_k)\rightarrow S(K,u+\eps)$.  Therefore, for every $\eps>0$,
\begin{equation*}
    L
    \le
    \lim_{k\rightarrow\infty}2S(K,v_k)\exp[H(K)]
    =
    2S(K,u+\eps)\exp[H(K)].
\end{equation*}
For $0<\eps\le1$, monotone convergence applied to
$\log(C+1)-\log(u+\eps)$ gives
$S(K,u+\eps)\rightarrow S(K,u)$ as $\eps\downarrow0$, including when
$S(K,u)=0$.  As this bound holds for every $\eps>0$,
\begin{equation*}
    L
    \le
    \inf_{\eps>0}2S(K,u+\eps)\exp[H(K)]
    =
    2S(K,u)\exp[H(K)].
\end{equation*}
This is \eqref{CSZ-upper-input-eq}.
\end{proof}

\begin{theorem}\label{cheb-upper-thm}
Let $K=K_0\cup X$ satisfy \eqref{support-main-assumptions}, and assume that $K_0$ is Parreau--Widom and $\BC(K)<\infty$.  If $w:K\rightarrow[0,\infty)$ is upper semicontinuous and $\supp(w)=K$, then
\begin{equation}\label{cheb-upper-eq}
    \limsup_{n\rightarrow\infty} W_{\infty,n}(K,w)
    \le
    2S(K,w)\exp[\PW(K)+\BC(K)].
\end{equation}
\end{theorem}

\begin{proof}
The set $K$ is semi-regular and $g_K=g_{K_0}$.  Since $K$ is compact, the
upper semicontinuous weight $w$ is
bounded.  By Proposition~\ref{CSZ-upper-input}, it is enough to prove
\begin{equation}\label{gap-sum-PW-BC}
    \sum_G \sup_{x\in G}g_K(x)
    \le
    \PW(K)+\BC(K),
\end{equation}
where $G$ ranges over the bounded components of $\cvh(K)\bs K$.

Let $(\alpha_\ell,\beta_\ell)$ be a bounded gap of $K_0$, and let $c_\ell$ be
the critical point of $g_K$ in this gap.  On $(\alpha_\ell,c_\ell)$ the Green
function is strictly increasing, and on $(c_\ell,\beta_\ell)$ it is strictly
decreasing.  Hence, after the points of $X$ have been inserted into this gap,
the supremum over each component of $(\alpha_\ell,\beta_\ell)\bs X$ is attained
either at $c_\ell$ if that component contains $c_\ell$, or as the limiting
value at one of the adjacent points of $X$.  The critical value $g_K(c_\ell)$
is counted at most once unless $c_\ell\in X$, in which case the two adjacent
components both have supremum $g_K(c_\ell)$.  These two possible contributions
are bounded by the single Parreau--Widom term $g_K(c_\ell)$ and the single
Blaschke term associated with the isolated point $c_\ell$.  Every isolated point
in the gap different from $c_\ell$ can be the maximizing endpoint of at most one
adjacent component.  Thus the total contribution of the components contained in
$(\alpha_\ell,\beta_\ell)$ is bounded by
$g_K(c_\ell)+\sum_{x_j\in(\alpha_\ell,\beta_\ell)}g_K(x_j)$.
On each exterior component of $\bb R\bs\cvh(K_0)$, the Green function is
monotone.  Consequently each isolated point outside $\cvh(K_0)$ can be the
maximizing endpoint of at most one bounded component of $\cvh(K)\bs K$.
Summing over all gaps of $K_0$ and over the exterior components gives
\eqref{gap-sum-PW-BC}.  The estimate \eqref{cheb-upper-eq} follows from
Proposition~\ref{CSZ-upper-input}.
\end{proof}

The next result transfers the Chebyshev upper bound to finite $p$.  The Gibbs
variational formula for the Szeg\H{o} integral treats the singular part on
$K_0$ and the isolated point masses simultaneously.

\begin{proposition}\label{Lp-majorant-lem}
Let $0<p<\infty$.  Let $K=K_0\cup X$ satisfy
\eqref{support-main-assumptions}, and let $d\mu$ be a finite positive Borel measure with $\supp(d\mu)=K$ given by \eqref{mu-main-decomp}.
For every $\eps>0$ there is a positive continuous weight $v$ on $K$ such that
\begin{equation}\label{majorant-norm-eq}
    \int |P|^p\,d\mu
    \le
    \|vP\|_K^p
\end{equation}
for every polynomial $P$.  Moreover,
\begin{equation}\label{majorant-S-eq}
    S(K,v)^p\le S(K,w)+\eps.
\end{equation}
\end{proposition}

\begin{proof}
Put $d\tau=d\sigma+\sum_j a_j\,d\delta_{x_j}$.  Since $\mu_K(X)=0$ and
$d\sigma\perp d\mu_K$, we have $d\mu=w\,d\mu_K+d\tau$ with
$d\tau\perp d\mu_K$.  After normalizing
$d\mu$ to a probability measure and scaling back, the Gibbs variational
formula applied to the pair $(\mu_K,\mu)$ with $f=e^g$ gives
\begin{equation}\label{majorant-variational}
    S(K,w)
    =
    \inf_{\substack{f\in C(K)\\ f>0}}
    \left(\int_K f\,d\mu\right)
    \exp\left(-\int_K\log f\,d\mu_K\right).
\end{equation}
See \cite[(2.2.4) and Proposition~10.6.3]{Sim11}.

Choose $f\in C(K)$, $f>0$, so that the product inside the infimum in
\eqref{majorant-variational} is at most $S(K,w)+\eps$.  Multiplying $f$ by a
positive constant does not change this product, so we may assume that
$\int_K f\,d\mu=1$.  Set $v=f^{-1/p}$.  Then $v$ is a positive continuous
weight on $K$, and for every polynomial $P$,
\begin{equation*}
    \int_K|P|^p\,d\mu
    =
    \int_K|vP|^p f\,d\mu
    \le
    \|vP\|_K^p.
\end{equation*}
Thus \eqref{majorant-norm-eq} holds.  Moreover,
$S(K,v)^p=\exp(-\int_K\log f\,d\mu_K)$.  Under the normalization above, this
is precisely the product chosen in \eqref{majorant-variational}, and hence
\eqref{majorant-S-eq} follows.
\end{proof}

The lower estimate rests on the fact that finite-$p$ extremal polynomials place
zeros near every isolated point mass.

\begin{proposition}\label{forced-zero-lem}
Let $0<p<\infty$.  Let $K=K_0\cup X$ satisfy
\eqref{support-main-assumptions}, and let $d\mu$ be a finite positive Borel measure with $\supp(d\mu)=K$ given by \eqref{mu-main-decomp}.  Assume that $S(K,w)>0$.  For each $n$, let $P_n$ be an
$L^p(d\mu)$ extremal monic polynomial of degree $n$.  If $F\subset X$ is
a finite set, then, for all sufficiently large $n$, one may choose for every
$x\in F$ a real zero $\zeta_{n,x}$ of $P_n$ such that $\zeta_{n,x}\rightarrow x$.
The zeros may be chosen distinct for different points of $F$.
\end{proposition}

\begin{proof}
Since $S(K,w)>0$, we have $\log w\in L^1(d\mu_K)$, and hence
$w>0$ $\mu_K$-a.e.  Recalling that $K=K_0\cup X\subset\bb R$ with $K_0$ regular and $X$ polar, we have $\ca(K)=\ca(K_0)$, $\mu_K=\mu_{K_0}$, and $\supp(d\mu_K)=K_0$.
By Erd\H{o}s--Tur\'an--Stahl--Totik criterion \cite[Theorem~4.1.1]{ST92}, the measure $d\eta:=wd\mu_K$ is regular. Moreover, $\supp(d\eta)=K_0$.

Let $T_n$ be the monic Chebyshev polynomial of degree $n$ on $K$.  Since
$P_n$ is $L^p(d\mu)$ extremal,
\[
    \|P_n\|_{L^p(d\mu)}
    \le
    \|T_n\|_{L^p(d\mu)}
    \le
    \mu(K)^{1/p}\|T_n\|_K.
\]
Szeg\H{o}'s root asymptotics \cite{Sze24}, also given in
\cite[Corollary~5.5.5]{Ran95},
$\|T_n\|_K^{1/n}\rightarrow\ca(K)$, imply
\begin{equation}\label{Lp-extremal-root-eq}
    \limsup_{n\rightarrow\infty}\|P_n\|_{L^p(d\mu)}^{1/n}\le\ca(K).
\end{equation}
Since $d\eta\le d\mu$, it follows that
\[
    \limsup_{n\rightarrow\infty}\|P_n\|_{L^p(d\eta)}^{1/n}\le\ca(K_0).
\]
The norm comparison theorem \cite[Theorem~3.4.3~(v)]{ST92} for polynomials of degree at most $n$ and regular measures with regular support then gives
\[
    \limsup_{n\rightarrow\infty}\|P_n\|_{K_0}^{1/n}\le\ca(K_0).
\]
This together with the Szeg\H{o} lower bound $\|P_n\|_{K_0}\ge\ca(K_0)^n$ shows that the polynomials $\{P_n\}$ are asymptotically extremal in the sup norm on $K_0$.

Since $\bb C\bs K_0$ is connected, the zero-distribution theorem of Blatt--Saff--Simkani \cite[Theorem~2.1]{BSS88} for asymptotically extremal polynomials gives weak convergence of the normalized zero counting measures to the equilibrium measure of $K_0$,
\begin{equation}\label{zero-counting-convergence-eq}
    \nu_n:=\frac1n\sum_{P_n(\zeta)=0}\delta_\zeta
    \xrightarrow{w^*}\mu_{K_0},
\end{equation}
where the zeros are counted with multiplicity, see also \cite[Theorem~III.4.1]{ST97}.

Next, we note that all zeros of $P_n$ lie in $\cvh(K)$, the convex hull of $K$.  Indeed, if $\zeta\notin\cvh(K)$ and $\zeta_0$ is the nearest point of $\cvh(K)$, then $q=\sup_{t\in K}\big|\frac{t-\zeta_0}{t-\zeta}\big|<1$.  Replacing the factor $z-\zeta$ by $z-\zeta_0$ therefore gives another monic polynomial whose absolute value on $K$ is at most $q|P_n|$.  This contradicts extremality.  Since $\cvh(K)\subset\bb R$, all zeros are real.

Suppose by contradiction that for some $x\in F$ there is a subsequence $n_j$ and an open interval $I\ni x$ such that $P_{n_j}$ has no zero in $I$ for every $j$.  By shrinking the interval $I$ we may assume that $\dist(I,K_0)>0$.  Then $\nu_{n_j}$ are supported on the compact set $J=\cvh(K)\bs I$. Since $t\mapsto\log|x-t|$ is continuous on $J$ and $K_0\subset J$,  \eqref{zero-counting-convergence-eq} gives
\begin{align}\label{point-value-asymptotic}
    \lim_{j\rightarrow\infty}\frac1{n_j}\log|P_{n_j}(x)|
    &=
    \lim_{j\rightarrow\infty}\int \log|x-t|\,d\nu_{n_j}(t) \notag\\
    &=
    \int\log|x-t|\,d\mu_{K_0}(t)
    =
    \log\ca(K_0)+g_{K_0}(x).
\end{align}
Here we have $g_{K_0}(x)>0$, since $x\notin K_0$. On the other hand,
\[
    \mu(\{x\})^{1/p}|P_n(x)| \le \|P_n\|_{L^p(\mu)}
\]
and since $\mu(\{x\})>0$ it follows from \eqref{Lp-extremal-root-eq} that
\[
    \limsup_{j\rightarrow\infty}\frac1{n_j}\log|P_{n_j}(x)|
    \le
    \log\ca(K) = \log\ca(K_0).
\]
This contradicts \eqref{point-value-asymptotic}.  Thus every neighborhood of $x$ contains a zero of $P_n$ for all large $n$.  Choose pairwise disjoint intervals $I_x$, $x\in F$.  For all large $n$, each $I_x$ contains a zero of $P_n$.  Let $\zeta_{n,x}$ be a zero in $I_x$ closest to $x$.  Since the same conclusion holds with $I_x$ replaced by any smaller
neighborhood of $x$, we have $\zeta_{n,x}\rightarrow x$.
\end{proof}

\begin{theorem}\label{main-estimates-thm}
Let $0<p\le\infty$ and let $K=K_0\cup X$ satisfy
\eqref{support-main-assumptions}.  For $0<p<\infty$, let $d\mu$ be a finite
positive Borel measure with $\supp(d\mu)=K$ given by \eqref{mu-main-decomp}.  When $p=\infty$, let $w$ be a bounded weight with $\supp(w)=K$ and interpret $W_{\infty,n}(K,\mu)$ as $W_{\infty,n}(K,w)$.
Assume that at least one of the two conditions
\begin{equation*}
    S(K,w)>0,
    \qquad
    \BC(K)<\infty
\end{equation*}
holds.  Then
\begin{equation}\label{main-lower-eq}
    \liminf_{n\rightarrow\infty}W_{p,n}(K,\mu)
    \ge
    S_p(K,w)e^{\BC(K)}.
\end{equation}

If, in addition, $K_0$ is Parreau--Widom, so that $\PW(K)<\infty$, and
$w$ is upper semicontinuous on $K$ when $p=\infty$, then
\begin{equation}\label{main-upper-eq}
    \limsup_{n\rightarrow\infty}W_{p,n}(K,\mu)
    \le
    2e^{\PW(K)+\BC(K)}S_p(K,w).
\end{equation}
\end{theorem}

\begin{proof}
We first prove the lower estimate.  If $S(K,w)=0$, then the hypothesis gives
$\BC(K)<\infty$, and the right-hand side of \eqref{main-lower-eq} is zero.
Hence, for the lower estimate, assume that $S(K,w)>0$.

Let $0<p<\infty$, and let $P_n$ be an $L^p(d\mu)$ extremal monic polynomial.
Fix a finite set $F\subset X$.  By Proposition~\ref{forced-zero-lem}, for all
large $n$ there are distinct zeros $\zeta_{n,x}$ of $P_n$, $x\in F$, such that
$\zeta_{n,x}\rightarrow x$.  Hence, with the zeros in the sum on the left
counted according to multiplicity,
\begin{equation}\label{selected-zero-sum-eq}
    \sum_{P_n(\zeta)=0}g_K(\zeta)
    \ge \sum_{x\in F}g_K(\zeta_{n,x}).
\end{equation}
Indeed, the $\zeta_{n,x}$ are distinct zeros of $P_n$, so the sum on the left
contains every term in the sum on the right, while all its remaining terms are
nonnegative since $g_K\ge0$.  Jensen's inequality, \eqref{GrFn}, and
\eqref{selected-zero-sum-eq} now give
\begin{align*}
    t_{p,n}(K,\mu)^p
    &\ge
    \int_{K_0}|P_n|^p w\,d\mu_K\\
    &\ge
    S(K,w)\exp\left(p\int_{K_0}\log|P_n|\,d\mu_K\right)\\
    &=
    \ca(K)^{np}S(K,w)
    \exp\left(p\sum_{P_n(\zeta)=0}g_K(\zeta)\right)\\
    &\ge
    \ca(K)^{np}S(K,w)
    \exp\left(p\sum_{x\in F}g_K(\zeta_{n,x})\right).
\end{align*}
Since $g_K=g_{K_0}$ is continuous near every point of $F$, it follows that
\begin{equation*}
    \liminf_{n\rightarrow\infty}W_{p,n}(K,\mu)
    \ge
    S_p(K,w)\exp\left(\sum_{x\in F}g_K(x)\right).
\end{equation*}
Taking the supremum over finite $F\subset X$ proves \eqref{main-lower-eq} for
$0<p<\infty$.

For $p=\infty$, let $1\le q<\infty$ and set
\begin{equation*}
    d\nu_q
    =
    w^q\,d\mu_K+\sum_j 2^{-j}w(x_j)^q\,d\delta_{x_j}.
\end{equation*}
Since $w$ is bounded on $K$, the measure $d\nu_q$ is finite.  Since
$S(K,w)>0$, we have $w>0$ $\mu_K$-a.e., so $w^q\,d\mu_K$ has support
$K_0$.  Moreover, since $\supp(w)=K$ and every $x_j$ is isolated in $K$,
we have $w(x_j)>0$.  Thus every $x_j\in X$ is an atom of $d\nu_q$ and
$\supp(d\nu_q)=K$.  Since $\mu_K(K_0)=1$ and
$|wP|\le\|wP\|_K$ on $K$, every monic polynomial $P$ satisfies
\begin{align*}
    \left(\int_K |P|^q\,d\nu_q\right)^{1/q}
    &=
    \left(
        \int_{K_0}|wP|^q\,d\mu_K
        +\sum_j 2^{-j}|w(x_j)P(x_j)|^q
    \right)^{1/q}\\
    &\le
    \left(1+\sum_j 2^{-j}\right)^{1/q}\|wP\|_K\\
    &\le
    2^{1/q}\|wP\|_K.
\end{align*}
The absolutely continuous part of $d\nu_q$ has density $w^q$, and hence
$S_q(K,w^q)=S(K,w)$.  Applying the already proved case $0<p<\infty$ of
\eqref{main-lower-eq} to the measure $d\nu_q$ with $p=q$ gives
\begin{equation*}
    2^{1/q}\liminf_{n\rightarrow\infty}W_{\infty,n}(K,w)
    \ge
    \liminf_{n\rightarrow\infty}W_{q,n}(K,\nu_q)
    \ge
    S(K,w)e^{\BC(K)}.
\end{equation*}
Letting $q\rightarrow\infty$ proves the lower estimate for $p=\infty$.

We now prove the upper estimate.  If $\BC(K)=\infty$, then the standing hypothesis gives $S(K,w)>0$, so the right-hand side of
\eqref{main-upper-eq} is $+\infty$ and there is nothing to prove.  Hence assume that $\BC(K)<\infty$.  For $p=\infty$, the assertion follows directly from Theorem~\ref{cheb-upper-thm}.

Let $0<p<\infty$.  Given $\eps>0$, choose $v$ from
Proposition~\ref{Lp-majorant-lem}.  By \eqref{majorant-norm-eq}, for every $n$,
\begin{equation*}
    t_{p,n}(K,\mu)^p
    =
    \inf_{\substack{P\ {\rm monic}\\ \deg P=n}}
    \int_K|P|^p\,d\mu
    \le
    \inf_{\substack{P\ {\rm monic}\\ \deg P=n}}
    \|vP\|_K^p
    =t_n(K,v)^p.
\end{equation*}
Taking $p$th roots and dividing by $\ca(K)^n$ gives
$W_{p,n}(K,\mu)\le W_{\infty,n}(K,v)$.  Therefore
Theorem~\ref{cheb-upper-thm} and \eqref{majorant-S-eq} yield
\begin{align*}
    \limsup_{n\rightarrow\infty}W_{p,n}(K,\mu)
    &\le
    \limsup_{n\rightarrow\infty}W_{\infty,n}(K,v)\\
    &\le
    2e^{\PW(K)+\BC(K)}S(K,v)\\
    &\le
    2e^{\PW(K)+\BC(K)}\bigl(S(K,w)+\eps\bigr)^{1/p}.
\end{align*}
Letting $\eps\downarrow0$ proves \eqref{main-upper-eq}.
\end{proof}

\section{Szeg\H{o} Theorems}\label{sec-szego-theorems}

For probability measures $\rho$ and $\tau$ on a compact interval, define
the entropy functional by
$D(\rho\mid\tau)=-\int\log\!\left(\frac{d\rho}{d\tau}\right)\,d\rho$
when $\rho\ll\tau$, and put
$D(\rho\mid\tau)=-\infty$ otherwise.  This is the quantity denoted by
$S(\rho\mid\tau)$ in \cite{Sim11}.  It satisfies $D(\rho\mid\tau)\le0$ and
is jointly upper semicontinuous under weak convergence
\cite[Proposition~2.2.2 and Theorem~2.2.3, pp.~49--50]{Sim11}.

\subsection{Consequences of the upper and lower bounds}

We use the conditions introduced in \eqref{intro-four-properties}.  The lower
and upper estimates in Theorem~\ref{main-estimates-thm} have different
hypotheses, so we first record their consequences separately.

\begin{theorem}\label{szego-implications-thm}
Let $0<p\le\infty$ and let $K=K_0\cup X$ satisfy
\eqref{support-main-assumptions}. For $0<p<\infty$, let $d\mu$ be a finite positive Borel measure with $\supp(d\mu)=K$ given by \eqref{mu-main-decomp}.  When $p=\infty$, let $w$ be a bounded weight with support $K$.
\begin{enumerate}[\quad$(a)$]
\item If $(S)$ holds, then $(L_p)$ holds.  If, in addition, $(B)$ fails, then
$W_{p,n}\rightarrow\infty$.  Consequently,
$(S)+(U_p)\Rightarrow(B)$.
\item Assume additionally, when $p=\infty$, that $w$ is upper semicontinuous
on $K$.  If $(P)$ and $(B)$ hold, then $(U_p)$ holds.  If, in addition,
$(S)$ fails, then $W_{p,n}\rightarrow0$.  Consequently,
$(P)+(B)+(L_p)\Rightarrow(S)$.
\end{enumerate}
\end{theorem}

\begin{proof}
Assume $(S)$.  The lower estimate \eqref{main-lower-eq} gives
\begin{equation*}
    \ul W_p(K,\mu)\ge S_p(K,w)e^{\BC(K)}.
\end{equation*}
The right-hand side is positive when $\BC(K)<\infty$ and is $+\infty$ when
$\BC(K)=\infty$.  Thus $(L_p)$ always holds, and failure of $(B)$ gives
$W_{p,n}(K,\mu)\rightarrow\infty$.  The latter conclusion is incompatible
with $(U_p)$, proving the final assertion in (a).

Now assume $(P)+(B)$.  The upper estimate \eqref{main-upper-eq} gives
\begin{equation*}
    \ol W_p(K,\mu)
    \le 2e^{\PW(K)+\BC(K)}S_p(K,w)<\infty,
\end{equation*}
so $(U_p)$ holds.  If $(S)$ fails, the right-hand side is zero.  Since the
Widom factors are nonnegative, $W_{p,n}(K,\mu)\rightarrow0$.  This is
incompatible with $(L_p)$ and proves the final assertion in (b).
\end{proof}

\begin{corollary}\label{3szego-thm}
Assume $(P)$ and the hypotheses of
Theorem~\ref{szego-implications-thm}.  When $p=\infty$, assume also that $w$ is
upper semicontinuous on $K$.  If at least one of $(B)$ and $(S)$ holds,
exactly one of the following three cases occurs:
\begin{enumerate}[\quad$(a)$]
\item $(B)$ and $(S)$ hold, and then both $(U_p)$ and $(L_p)$ hold.
\item $(B)$ holds and $(S)$ fails, and then
$W_{p,n}(K,\mu)\rightarrow0$.
\item $(S)$ holds and $(B)$ fails, and then
$W_{p,n}(K,\mu)\rightarrow\infty$.
\end{enumerate}
The conclusion in {\rm(c)} requires neither $(P)$ nor, when $p=\infty$,
upper semicontinuity of $w$.  Under $(P)+(B)$, the conditions $(S)$, $(L_p)$,
and $(W_p)$ are equivalent, while under $(P)+(S)$, the conditions $(B)$,
$(U_p)$, and $(W_p)$ are equivalent.
\end{corollary}

\begin{proof}
The three cases follow from the two parts of
Theorem~\ref{szego-implications-thm}.  Under $(P)+(B)$, cases (a) and (b)
give the first set of equivalences.  Under $(P)+(S)$, cases (a) and (c) give
the second.
\end{proof}

For either choice $W_p^\star\in\{\ul W_p,\ol W_p\}$, the same trichotomy shows
that, under $(P)$, any two of $(B)$, $(S)$, and
$0<W_p^\star(K,\mu)<\infty$ imply the third.

\subsection{A converse estimate for the Parreau--Widom sum}

In this subsection we prove that, for $2\le p\le\infty$,
$(S)+(U_p)\Rightarrow(P)$ and establish
the quantitative estimate \eqref{unconditional-pw-bound}. We first bound
the Parreau--Widom sum by the relative entropy of $\mu_K$ with respect to
normalized Lebesgue measure on $\cvh(K_0)$.  We then approximate this entropy by dyadic
partitions of $K_0$ into sets of equal $\mu_K$-measure.  These partition
estimates will ultimately give a lower bound for averages of the logarithms
of the $L^2$ Widom factors.

\begin{lemma}\label{equilibrium-entropy-bound}
Let $[A,B]=\cvh(K_0)$, let $L=B-A$, and let $\lambda$ be normalized
Lebesgue measure on $[A,B]$.  Then
\begin{equation}\label{equilibrium-entropy-bound-eq}
    -D(\mu_K\mid\lambda)
    \ge \PW(K)+\log\frac{L}{\pi\ca(K)}.
\end{equation}
\end{lemma}

\begin{proof}
If $D(\mu_K\mid\lambda)=-\infty$, there is nothing to prove.  Otherwise,
$\mu_K$ is absolutely continuous with respect to Lebesgue measure.  Enumerate
the bounded gaps of $K_0$ and put
$K_{0,m}=[A,B]\bs\bigcup_{j=1}^m(\alpha_j,\beta_j)$, where $m$ does not
exceed the number of gaps and the empty union is allowed.  If $c_{j,m}$ is
the critical point of $g_{K_{0,m}}$ in the $j$th gap, then the finite-gap
equilibrium density is
\begin{equation}\label{finite-gap-density}
    \frac{d\mu_{K_{0,m}}}{dx}(x)
    =\frac1\pi
      \frac{\prod_{j=1}^m|x-c_{j,m}|}
      {\sqrt{|(x-A)(x-B)\prod_{j=1}^m
      (x-\alpha_j)(x-\beta_j)|}}
\end{equation}
for Lebesgue-a.e. $x\in K_{0,m}$ \cite[(5.4.96), p.~280 and
Theorem~5.5.22(iii)--(iv), pp.~302--303]{Sim11}.  This density is positive
Lebesgue-a.e. on $K_{0,m}$.  Since $\mu_K$ is supported on $K_0$ and is
absolutely continuous with respect to Lebesgue measure,
$\mu_K\ll\mu_{K_{0,m}}$.  Every endpoint in
\eqref{finite-gap-density} belongs to $K_0$.  Since
$\frac{d\mu_{K_{0,m}}}{d\lambda}
=L\frac{d\mu_{K_{0,m}}}{dx}$ $\lambda$-a.e., \eqref{GrFn} gives
\begin{equation}\label{finite-gap-log-density}
    \int_{K_0}\log\frac{d\mu_{K_{0,m}}}{d\lambda}\,d\mu_K
    =\log\frac{L}{\pi\ca(K)}+\sum_{j=1}^m g_K(c_{j,m}).
\end{equation}
Since $\mu_K\ll\mu_{K_{0,m}}\ll\lambda$, we have
$\frac{d\mu_K}{d\lambda}
=\frac{d\mu_K}{d\mu_{K_{0,m}}}
 \frac{d\mu_{K_{0,m}}}{d\lambda}$ $\mu_K$-a.e.  The definition of $D$ and
\eqref{finite-gap-log-density} therefore give
\begin{equation*}
    D(\mu_K\mid\mu_{K_{0,m}})
    =D(\mu_K\mid\lambda)+\log\frac{L}{\pi\ca(K)}
      +\sum_{j=1}^m g_K(c_{j,m})\le0.
\end{equation*}
Therefore
\begin{equation}\label{finite-gap-entropy-bound}
    -D(\mu_K\mid\lambda)
    \ge \log\frac{L}{\pi\ca(K)}+\sum_{j=1}^m g_K(c_{j,m}).
\end{equation}
If $K_0$ has finitely many gaps, take all of them in
\eqref{finite-gap-entropy-bound}.  Otherwise,
\cite[Corollary~4.4.5 and Theorem~4.4.6, pp.~108--109]{Ran95} gives
$g_{K_{0,m}}\uparrow g_K$ on every fixed gap.  Harnack's theorem makes this
convergence locally uniform there
\cite[Theorem~1.3.9, p.~16]{Ran95}.  Fix $j$.  Since $g_K$ vanishes at the
endpoints of the $j$th gap and $g_K(c_j)>0$, there is a compact interval
$I_j$ contained in the gap such that $g_K<g_K(c_j)/2$ outside $I_j$.  Since
$g_{K_{0,m}}\le g_K$ and
$g_{K_{0,m}}(c_j)\rightarrow g_K(c_j)$, the maximizer $c_{j,m}$ belongs to
$I_j$ for all sufficiently large $m$.  Local uniform convergence on $I_j$
and the uniqueness of the maximizer of $g_K$ give
$c_{j,m}\rightarrow c_j$.
For $m\ge r$ with $r\ge1$ fixed, retain only the first $r$ terms in
\eqref{finite-gap-entropy-bound} and let $m\rightarrow\infty$.  Since
$c_{j,m}\rightarrow c_j$ for $1\le j\le r$, we obtain
\begin{equation*}
    -D(\mu_K\mid\lambda)
    \ge \log\frac{L}{\pi\ca(K)}
       +\sum_{j=1}^r g_K(c_j).
\end{equation*}
Then letting $r\rightarrow\infty$ proves \eqref{equilibrium-entropy-bound-eq}.
\end{proof}

Let $F(x)=\mu_K([A,x])$ on $[A,B]$.  For a finite Borel measure $\rho$ on
$[A,B]$, let $F_*\rho$ be the measure on $[0,1]$ defined by
$(F_*\rho)(Y)=\rho(F^{-1}(Y))$ for every Borel set $Y\subset[0,1]$.
Since $\mu_K$ has finite logarithmic energy, it does not charge polar sets
\cite[Theorem~3.2.3, p.~56]{Ran95}.  In particular, it has no atoms, so $F$
is continuous.  Since $F(A)=0$ and $F(B)=1$, it maps $[A,B]$ onto $[0,1]$.
Define the right-endpoint quantile map $q:[0,1]\rightarrow[A,B]$ by
\[
q(t)=\max F^{-1}(\{t\}), \quad 0\le t\le1.
\]
The level set $F^{-1}(\{t\})$ is nonempty and compact, so $q(t)$ is
well defined. The non-singleton fibers of $F$ are precisely the closures of the bounded gaps of $K_0$. On such a fiber, $q$ selects its right endpoint.
Then $F^{-1}([0,t])=[A,q(t)]$ and
$(F_*\mu_K)([0,t])=\mu_K([A,q(t)])=F(q(t))=t$.  Since finite Borel measures
on $[0,1]$ are determined by the intervals $[0,t]$, $F_*\mu_K$ is Lebesgue
measure on $[0,1]$.  If $s<t$, then $q(s)<q(t)$, so $q$ is increasing and
therefore Borel.  For $N\ge2$, put
\begin{equation*}
    J_{1,N}=[0,1/N],
    \qquad
    J_{k,N}=((k-1)/N,k/N],\quad 2\le k\le N,
\end{equation*}
and put
\begin{equation*}
    q_{0,N}=A,
    \qquad
    q_{N,N}=B,
    \qquad
    q_{k,N}=q(k/N),\quad 1\le k\le N-1.
\end{equation*}
Set
\begin{equation*}
\begin{aligned}
    K_{k,N}&=K_0\cap F^{-1}(J_{k,N}),
    &\ell_{k,N}&=q_{k,N}-q_{k-1,N}.
\end{aligned}
\end{equation*}
We have $F^{-1}(J_{1,N})=[A,q_{1,N}]$ and
$F^{-1}(J_{k,N})=(q_{k-1,N},q_{k,N}]$ for $k\ge2$.  Consequently,
$\ell_{k,N}>0$ and
$K_{k,N}\subset[q_{k-1,N},q_{k,N}]$.  The definition of the pushforward
and the fact that $F_*\mu_K$ is Lebesgue measure give
\begin{equation}\label{partition-pushforward-masses}
    (F_*\mu_K)(J_{k,N})=\mu_K(K_{k,N})=\frac1N,
    \qquad
    (F_*\lambda)(J_{k,N})=\frac{\ell_{k,N}}L.
\end{equation}
Define
\begin{equation}\label{partition-entropy-def}
\begin{aligned}
    H_N
    &=\sum_{k=1}^N(F_*\mu_K)(J_{k,N})
      \log\frac{(F_*\mu_K)(J_{k,N})}
      {(F_*\lambda)(J_{k,N})}\\
    &=\frac1N\sum_{k=1}^N
      \log\frac{1/N}{\ell_{k,N}/L}.
\end{aligned}
\end{equation}

\begin{lemma}\label{partition-entropy-lem}
Along $N=2^m$,
\begin{equation}\label{partition-entropy-limit}
    H_{2^m}\uparrow-D(\mu_K\mid\lambda).
\end{equation}
\end{lemma}

\begin{proof}
We first show that mapping $[A,B]$ to $[0,1]$ by $F$ preserves relative
entropy, even though every bounded gap is collapsed to a point.  We then
show that $H_{2^m}$ is the relative entropy seen at the level of the
dyadic partition $\{J_{k,2^m}\}_{k=1}^{2^m}$ and let the mesh of these
partitions tend to zero.

Let $(\alpha_\ell,\beta_\ell)$ range over the bounded gaps of $K_0$, and put
$\mathcal P=\{\mu_K([A,\alpha_\ell])\}_\ell$.  These are the harmonic measure
frequencies associated with the gaps, measured from the left.  Since
$\supp(d\mu_K)=K_0$ and $\mu_K$ has no atoms, we have
$\mathcal P\subset(0,1)$,
$F^{-1}(\mathcal P)=\bigcup_\ell[\alpha_\ell,\beta_\ell]$, and
$\mu_K(F^{-1}(\mathcal P))=0$.

Suppose first that $\mu_K\ll\lambda$, and let
$h=\frac{d\mu_K}{d\lambda}$.  Since $\mu_K(F^{-1}(\mathcal P))=0$, we may
take $h=0$ on $F^{-1}(\mathcal P)$.  Put
$\widetilde h=h\circ q$.  If $x\notin F^{-1}(\mathcal P)$, then
$q(F(x))=x$.  If $x\in F^{-1}(\mathcal P)$, then
$q(F(x))\in F^{-1}(\mathcal P)$, so $h(x)=h(q(F(x)))=0$.  Hence
$h=\widetilde h\circ F$.  Therefore, for every Borel set $Y\subset[0,1]$,
$(F_*\mu_K)(Y)=\mu_K(F^{-1}(Y))
=\int_{F^{-1}(Y)}h(x)\,d\lambda(x)$.  Substituting
$h=\widetilde h\circ F$ gives
$\int_{F^{-1}(Y)}\widetilde h(F(x))\,d\lambda(x)$.  By the integration
formula for the pushforward measure, this is
$\int_Y\widetilde h(t)\,d(F_*\lambda)(t)$.  Thus
$F_*\mu_K\ll F_*\lambda$ and
$\frac{d(F_*\mu_K)}{d(F_*\lambda)}=\widetilde h$.  Since
$h=\widetilde h\circ F$, the pushforward integration formula gives
$-\int_{[0,1]}\log\widetilde h\,d(F_*\mu_K)
=-\int_{[A,B]}\log(\widetilde h\circ F)\,d\mu_K
=-\int_{[A,B]}\log h\,d\mu_K$.  Therefore,
\begin{equation}\label{entropy-under-F}
    D(F_*\mu_K\mid F_*\lambda)=D(\mu_K\mid\lambda).
\end{equation}
If $\mu_K\not\ll\lambda$, choose a Borel set
$Z\subset[A,B]\bs F^{-1}(\mathcal P)$ such that
$\lambda(Z)=0<\mu_K(Z)$.  For every $x\in Z$,
$F^{-1}(\{F(x)\})=\{x\}$.  Hence
$F(Z)=q^{-1}(Z)$ is Borel and $F^{-1}(F(Z))=Z$.  Therefore
$(F_*\lambda)(F(Z))=0<(F_*\mu_K)(F(Z))$, so
\eqref{entropy-under-F} again holds with both sides equal to $-\infty$.

It remains to approximate this common entropy by the dyadic partitions.
For the pair $(F_*\mu_K,F_*\lambda)$, the variational formula for relative
entropy gives
\begin{equation}\label{partition-entropy-variational}
    D(F_*\mu_K\mid F_*\lambda)
    =
    \inf_{\substack{f\in C([0,1])\\ f>0}}\Psi(f),
\end{equation}
where
\begin{equation*}
    \Psi(f)=\int f\,d(F_*\lambda)
      -\int(1+\log f)\,d(F_*\mu_K).
\end{equation*}
See \cite[Proposition~2.2.2, (2.2.5)--(2.2.6), pp.~49--50]{Sim11}.

We first minimize the functional $\Psi$ over functions that are constant
on each dyadic cell.
For $m\ge1$, let $\mathcal E_m$ be the positive functions on $[0,1]$ that
are constant on each interval $J_{k,2^m}$.
If $f=c>0$ on $J_{k,2^m}$, then
\eqref{partition-pushforward-masses} shows that the contribution of this
interval to $\Psi(f)$ is
\begin{equation*}
    c\frac{\ell_{k,2^m}}L-\frac1{2^m}(1+\log c).
\end{equation*}
Its first and second derivatives are
$\ell_{k,2^m}/L-1/(2^m c)$ and $1/(2^m c^2)$, respectively.  Hence its
minimum is attained at $c=L/(2^m\ell_{k,2^m})$, where its value is
\begin{equation*}
    -\frac1{2^m}
      \log\frac{1/2^m}{\ell_{k,2^m}/L}.
\end{equation*}
Summing over $k$ and using \eqref{partition-entropy-def} gives
\begin{equation}\label{partition-functional-minimum}
    \inf_{f\in\mathcal E_m}\Psi(f)=-H_{2^m}.
\end{equation}
Since $\mathcal E_m\subset\mathcal E_{m+1}$, the sequence $H_{2^m}$ is
increasing.  Suppose that $F_*\mu_K\ll F_*\lambda$, and put
$r=d(F_*\mu_K)/d(F_*\lambda)$.  Fix $1\le k\le2^m$ and write
$J=J_{k,2^m}$.  By \eqref{partition-pushforward-masses},
the measure $L\ell_{k,2^m}^{-1}\chi_J\,d(F_*\lambda)$ has total mass one,
and
\begin{equation*}
    \frac{L}{\ell_{k,2^m}}
      \int_J r\,d(F_*\lambda)
    =\frac{L}{\ell_{k,2^m}}(F_*\mu_K)(J)
    =\frac{L}{2^m\ell_{k,2^m}}.
\end{equation*}
Jensen's inequality for $x\log x$ therefore gives
\begin{equation*}
\begin{aligned}
    \frac{L}{\ell_{k,2^m}}
      \int_J r\log r\,d(F_*\lambda)
    &\ge
      \left(\frac{L}{\ell_{k,2^m}}
      \int_J r\,d(F_*\lambda)\right)
      \log\left(\frac{L}{\ell_{k,2^m}}
      \int_J r\,d(F_*\lambda)\right)\\
    &=\frac{L}{2^m\ell_{k,2^m}}
      \log\frac{L}{2^m\ell_{k,2^m}}.
\end{aligned}
\end{equation*}
Multiplication by $\ell_{k,2^m}/L$ yields
\begin{equation*}
    \int_J r\log r\,d(F_*\lambda)
    \ge\frac1{2^m}
      \log\frac{1/2^m}{\ell_{k,2^m}/L}.
\end{equation*}
After summing over $k$, \eqref{partition-entropy-def} and the definition of
$D$ give
\begin{equation}\label{partition-entropy-upper}
    D(F_*\mu_K\mid F_*\lambda)
    =-\int_{[0,1]}r\log r\,d(F_*\lambda)
    \le-H_{2^m}.
\end{equation}
If $F_*\mu_K\not\ll F_*\lambda$, then
$D(F_*\mu_K\mid F_*\lambda)=-\infty$, so
\eqref{partition-entropy-upper} still holds.

Conversely, let $f$ be a positive continuous function on $[0,1]$ (and hence
bounded away from zero), and choose $f_m\in\mathcal E_m$ such that
$f_m\rightarrow f$ uniformly.  Then $\log f_m\rightarrow\log f$ uniformly, and
\eqref{partition-functional-minimum} gives
$-H_{2^m}\le\Psi(f_m)\rightarrow\Psi(f)$.  Taking the infimum over all positive continuous $f$ and using \eqref{partition-entropy-variational} gives
\begin{equation*}
    \lim_{m\rightarrow\infty}(-H_{2^m})
    \le D(F_*\mu_K\mid F_*\lambda).
\end{equation*}
Together with \eqref{partition-entropy-upper}, this proves
$H_{2^m}\uparrow-D(F_*\mu_K\mid F_*\lambda)$.  Equation
\eqref{entropy-under-F} proves \eqref{partition-entropy-limit}.
\end{proof}

\begin{theorem}\label{unconditional-pw-thm}
Let $2\le p\le\infty$ and assume the hypotheses of
Theorem~\ref{szego-implications-thm}.  If $(S)$ and $(U_p)$ hold, then
$(P)$ and $(B)$ hold.  More precisely,
\begin{equation}\label{unconditional-pw-bound}
    \PW(K)
    \le 2\log\frac{\ol W_p(K,\mu)}{S_p(K,w)}-\log\frac4\pi.
\end{equation}
\end{theorem}

\begin{proof}
We first prove \eqref{unconditional-pw-bound} for $p=2$.  Multiplying
$d\mu$ by a positive constant leaves the right-hand side unchanged, so we
may assume that $\mu(K)=1$.  In particular, $t_{2,0}(K,\mu)^2=\mu(K)=1$.

Our goal is to use the partition entropy $H_N$ to obtain a lower bound for
the logarithmic average
\begin{equation*}
    \frac2N\sum_{n=1}^{N-1}\log W_{2,n}(K,\mu).
\end{equation*}
The standard moment-matrix identity below provides the connection.  For a
finite positive measure $\rho$ with infinite compact support, let
$t_{2,n}(\rho)$ denote its monic $L^2(\rho)$ extremal norm and define its
$N\times N$ moment matrix by
\begin{equation*}
    G_N(\rho)=
    \left[\int x^{j+k}\,d\rho(x)\right]_{j,k=0}^{N-1}.
\end{equation*}
Gram--Schmidt and Heine's multiple-integral formula
\cite[(2.2.7), (2.2.11), and (2.2.15), pp.~27--28]{Sze75} give
\begin{equation}\label{Gram-Heine-main}
    \det G_N(\rho)
    =\prod_{n=0}^{N-1}t_{2,n}(\rho)^2
    =\frac1{N!}\int_{\bb R^N}
      \prod_{1\le i<j\le N}(x_i-x_j)^2
      \prod_{k=1}^N d\rho(x_k).
\end{equation}
Here $x_1,\ldots,x_N$ are the coordinates on $\bb R^N$, $\prod_{k=1}^N d\rho(x_k)$ denotes the $N$-fold product measure, and the product $\prod_{1\le i<j\le N}(x_i-x_j)^2$ contains one factor for each pair
$1\le i<j\le N$.  Below, products and sums with the subscript $i<j$
are understood to range over these pairs.

The first identity in \eqref{Gram-Heine-main} expresses the determinant as
a product of the $L^2$ extremal norms of degrees $0,\ldots,N-1$.  It will
therefore convert a lower bound for the determinant into a lower bound for
the logarithmic average above.  The second identity allows us to obtain
such a determinant bound by restricting the integral to configurations
having one point in each of the quantile cells
$K_{1,N},\ldots,K_{N,N}$.

Each cell has $\mu_K$-measure $1/N$, so normalizing the restrictions of
$\mu_K$ to these cells produces a factor $N^{-N}$.  Jensen's inequality
then separates the contribution of the weight from the pairwise terms
$\log|x_i-x_j|$.  The cross-cell terms are recovered from the total
logarithmic energy of $\mu_K$, while the within-cell terms are bounded
using the capacities of the intervals
$[q_{k-1,N},q_{k,N}]$.  Their lengths $\ell_{k,N}$ are precisely the
quantities appearing in the partition entropy $H_N$.

For $1\le k\le N$, put $d\nu_{k,N}=N\chi_{K_{k,N}}\,d\mu_K$.  By
\eqref{partition-pushforward-masses}, these are probability measures.  Apply
\eqref{Gram-Heine-main} to $w\,d\mu_K$ and restrict its integral to the
$N$-tuples with one coordinate in each of $K_{1,N},\ldots,K_{N,N}$.  Since
the coordinates are labeled, this is the disjoint union of
$K_{\pi(1),N}\times\cdots\times K_{\pi(N),N}$ over the $N!$ permutations
$\pi$.  The squared Vandermonde and the product measure are symmetric in the
coordinates, so all these integrals are equal and the factor $1/N!$ cancels.
On $K_{k,N}$, we have $d\mu_K=N^{-1}d\nu_{k,N}$, and we use
$\prod_k w(x_k)\prod_{i<j}|x_i-x_j|^2
=\exp(\sum_k\log w(x_k)+2\sum_{i<j}\log|x_i-x_j|)$.
Since $d\nu_{k,N}$ is supported on $K_{k,N}$, we obtain
\begin{align}\label{retained-Gram-Heine}
    \det G_N(w\,d\mu_K)
    &\ge \frac1{N!}\sum_{\pi}
      \int_{K_{\pi(1),N}\times\cdots\times K_{\pi(N),N}}
      \prod_{i<j}|x_i-x_j|^2
      \prod_{k=1}^N\bigl(w(x_k)\,d\mu_K(x_k)\bigr)\notag\\
    &=\int_{K_{1,N}\times\cdots\times K_{N,N}}
      \prod_{i<j}|x_i-x_j|^2
      \prod_{k=1}^N\bigl(w(x_k)\,d\mu_K(x_k)\bigr)\notag\\
    &=N^{-N}\int_{\bb R^N}
      \exp\left(
      \sum_{k=1}^N\log w(x_k)
      +2\sum_{i<j}\log|x_i-x_j|\right)
      \prod_{k=1}^N d\nu_{k,N}(x_k).
\end{align}
Let
\begin{equation*}
    \Phi_N(x_1,\ldots,x_N)
    =\sum_{k=1}^N\log w(x_k)
      +2\sum_{i<j}\log|x_i-x_j|.
\end{equation*}
Each $\nu_{k,N}$ is a probability measure, so
$\nu_{1,N}\otimes\cdots\otimes\nu_{N,N}$ is also a probability measure.
We apply Jensen's inequality to the last integral in \eqref{retained-Gram-Heine}.
We have $\log w\in L^1(\mu_K)$ because
$\log^+w\le w$ and $(S)$ controls the negative part.  Since
$d\nu_{k,N}\le N\,d\mu_K$, this fact and the finite logarithmic energy of
$\mu_K$ show that $\Phi_N$ is integrable.  Moreover, $e^{\Phi_N}$ is
integrable because the Vandermonde factor is bounded on the compact set
$K^N$ and $w\in L^1(\mu_K)$.  The same integrand is positive almost
everywhere since $w>0$ $\mu_K$-a.e.\ and the sets $K_{1,N},\ldots,K_{N,N}$
are pairwise disjoint.  Thus Jensen's inequality gives
\begin{equation}\label{partition-Jensen}
\begin{aligned}
    \log\det G_N(w\,d\mu_K)
    &\ge -N\log N
      +\log\int_{\bb R^N}e^{\Phi_N}
        \prod_{k=1}^N d\nu_{k,N}(x_k)\\
    &\ge -N\log N
      +\int_{\bb R^N}\Phi_N
        \prod_{k=1}^N d\nu_{k,N}(x_k).
\end{aligned}
\end{equation}
Since each $\nu_{k,N}$ is a probability measure, integration in every
coordinate not occurring in a given summand contributes the factor one.
Thus Fubini's theorem and the definition of $\Phi_N$ give
\begin{equation}\label{partition-Phi-average}
\begin{aligned}
    \int_{\bb R^N}\Phi_N\prod_{k=1}^N d\nu_{k,N}(x_k)
    ={}&\sum_{k=1}^N\int_{\bb R}\log w(x_k)\,d\nu_{k,N}(x_k)\\
    &+2\sum_{i<j}\int_{\bb R}\int_{\bb R}\log|x_i-x_j|\,
      d\nu_{i,N}(x_i)d\nu_{j,N}(x_j).
\end{aligned}
\end{equation}
For $1\le i,j\le N$, put
$I_{ij}=\int_{K_{i,N}}\int_{K_{j,N}}
\log|x-y|\,d\mu_K(x)d\mu_K(y)$.
Using $d\nu_{k,N}=N\chi_{K_{k,N}}\,d\mu_K$ and the fact that the
$K_{k,N}$ partition $K_0$, we have
\begin{equation}\label{partition-weight-term}
\begin{aligned}
    \sum_{k=1}^N\int_{\bb R}\log w(x_k)\,d\nu_{k,N}(x_k)
    &=N\sum_{k=1}^N\int_{K_{k,N}}
      \log w(x_k)\,d\mu_K(x_k)\\
    &=N\int_{K_0}\log w(x)\,d\mu_K(x)=N\log S(K,w).
\end{aligned}
\end{equation}
For each $i<j$, the same definition gives
\begin{equation}\label{partition-interaction-term}
\begin{aligned}
    \int_{\bb R}\int_{\bb R}\log|x_i-x_j|\,
      d\nu_{i,N}(x_i)d\nu_{j,N}(x_j)
    &=N^2\int_{K_{i,N}}\int_{K_{j,N}}
      \log|x_i-x_j|\,d\mu_K(x_i)d\mu_K(x_j)\\
    &=N^2I_{ij}.
\end{aligned}
\end{equation}
Substitution of \eqref{partition-Phi-average},
\eqref{partition-weight-term}, and \eqref{partition-interaction-term} into
\eqref{partition-Jensen} gives
\begin{equation}\label{partition-determinant-intermediate}
    \log\det G_N(w\,d\mu_K)
    \ge -N\log N+N\log S(K,w)+2N^2\sum_{i<j}I_{ij}.
\end{equation}
Integrating \eqref{GrFn} against $d\mu_K$ and using $g_K=0$ on $K_0$ gives
\begin{equation*}
    \int_{K_0}\int_{K_0}\log|x-y|\,d\mu_K(x)d\mu_K(y)
    =\log\ca(K).
\end{equation*}
Since the sets $K_{k,N}$ partition $K_0$, it follows that
\begin{equation}\label{partition-energy-split}
    2\sum_{i<j}I_{ij}=\log\ca(K)-\sum_{k=1}^NI_{kk}.
\end{equation}
Since $K_{k,N}\subset[q_{k-1,N},q_{k,N}]$, the measure $d\nu_{k,N}$ is
supported on this interval, whose length is $\ell_{k,N}$.  As a probability
measure supported there, $\nu_{k,N}$ is a competitor for the maximal
logarithmic energy of the interval.  Since an interval of length
$\ell_{k,N}$ has capacity $\ell_{k,N}/4$
\cite[Corollary~5.2.4, p.~134]{Ran95},
\begin{equation}\label{partition-self-energy}
    N^2I_{kk}
    =\int\int\log|x-y|\,d\nu_{k,N}(x)d\nu_{k,N}(y)
    \le\log\frac{\ell_{k,N}}4.
\end{equation}
Combining \eqref{partition-determinant-intermediate},
\eqref{partition-energy-split}, and \eqref{partition-self-energy} gives
\begin{equation}\label{partition-determinant-lower}
    \log\det G_N(w\,d\mu_K)
    \ge -N\log N+N\log S(K,w)+N^2\log\ca(K)
      -\sum_{k=1}^N\log\frac{\ell_{k,N}}4.
\end{equation}

Since $d\mu\ge w\,d\mu_K$, the nonnegative multiple integral in
\eqref{Gram-Heine-main} gives
\begin{equation}\label{partition-Gram-domination}
    \det G_N(\mu)\ge\det G_N(w\,d\mu_K).
\end{equation}
Since $t_{2,0}(K,\mu)^2=\mu(K)=1$, the first identity in
\eqref{Gram-Heine-main} and the definition of $W_{2,n}(K,\mu)$ give
\begin{equation}\label{partition-Gram-Widom-product}
    \det G_N(\mu)=\ca(K)^{N(N-1)}
      \prod_{n=1}^{N-1}W_{2,n}(K,\mu)^2.
\end{equation}
Equations \eqref{partition-determinant-lower},
\eqref{partition-Gram-domination}, and
\eqref{partition-Gram-Widom-product} give
\begin{align}
    2\sum_{n=1}^{N-1}\log W_{2,n}(K,\mu)
    &\ge -N\log N+N\log S(K,w)+N\log\ca(K)
      -\sum_{k=1}^N\log\frac{\ell_{k,N}}4.\label{partition-Widom-intermediate}
\end{align}
Equation \eqref{partition-entropy-def} can be written as
\begin{equation}\label{partition-entropy-expanded}
    H_N=\log L-\log N-\frac1N\sum_{k=1}^N\log\ell_{k,N}.
\end{equation}
Dividing \eqref{partition-Widom-intermediate} by $N$ and then using
\eqref{partition-entropy-expanded} gives
\begin{align}\label{partition-Widom-bound}
    \frac2N\sum_{n=1}^{N-1}\log W_{2,n}(K,\mu)
    &\ge -\log N+\log S(K,w)+\log\ca(K)
      -\frac1N\sum_{k=1}^N\log\frac{\ell_{k,N}}4\notag\\
    &=\log S(K,w)+\log\frac{4\ca(K)}L+H_N.
\end{align}
Thus the determinant formula has converted the dyadic partition entropy
into a lower bound for a logarithmic average of the $L^2$ Widom factors.

Let $M=\ol W_2(K,\mu)$.  Applying Lemmas~\ref{equilibrium-entropy-bound} and \ref{partition-entropy-lem} to \eqref{partition-Widom-bound} along $N=2^m$ gives
\begin{equation}\label{partition-Widom-average-lower}
    \liminf_{m\rightarrow\infty}
    \frac{2}{2^m}\sum_{n=1}^{2^m-1}\log W_{2,n}(K,\mu)
    \ge \log S(K,w)+\PW(K)+\log\frac4\pi.
\end{equation}
Since $(S)$ holds, \eqref{main-lower-eq} gives $\ul W_2(K,\mu)>0$.  Together with $(U_2)$, this shows that $0<M<\infty$.  For every
$\eps>0$, all sufficiently large $n$ satisfy
$W_{2,n}(K,\mu)\le M+\eps$.  Since finitely many initial terms do not affect
these averages,
\begin{equation}\label{partition-Widom-average-upper}
    \limsup_{m\rightarrow\infty}
    \frac{2}{2^m}\sum_{n=1}^{2^m-1}\log W_{2,n}(K,\mu)
    \le 2\log(M+\eps).
\end{equation}
Combining \eqref{partition-Widom-average-lower} and
\eqref{partition-Widom-average-upper}, and then letting $\eps\downarrow0$,
proves
\eqref{unconditional-pw-bound} for $p=2$.

It remains to reduce the cases $2<p<\infty$ and $p=\infty$ to the
$p=2$ result. Let $2<p<\infty$, and put $d\rho_p(x)=w(x)^{2/p}d\mu_K(x)$.  H\"older's inequality with exponents $p/2$ and $p/(p-2)$ gives
\begin{equation*}
    \int_{K_0}w(x)^{2/p}\,d\mu_K(x)
    \le\left(\int_{K_0}w(x)\,d\mu_K(x)\right)^{2/p}
       \mu_K(K_0)^{1-2/p}<\infty.
\end{equation*}
Thus $\rho_p$ is a finite positive measure.  Since $w>0$ $\mu_K$-a.e.,
its support is $K_0$.
For every polynomial $Q$, Jensen's inequality gives
\begin{equation*}
    \|Q\|_{L^2(\rho_p)}^2
    =\int_{K_0}(|Q|^pw)^{2/p}\,d\mu_K
    \le\left(\int_{K_0}|Q|^pw\,d\mu_K\right)^{2/p}.
\end{equation*}
Thus $\|Q\|_{L^2(\rho_p)}\le\|Q\|_{L^p(\mu)}$ by
\eqref{mu-main-decomp}.
Since $\ca(K_0)=\ca(K)$, it follows that
\begin{equation*}
    \ol W_2(K_0,\rho_p)\le\ol W_p(K,\mu),
    \qquad
    S_2(K_0,w^{2/p})=S_p(K,w).
\end{equation*}
The $p=2$ result applied to $d\rho_p$ gives
\eqref{unconditional-pw-bound}.

For $p=\infty$, put $d\rho_\infty=w^2d\mu_K$.  This is a finite positive
measure because $w$ is bounded.  It has support $K_0$ because
$w>0$ $\mu_K$-a.e., and
$\|Q\|_{L^2(\rho_\infty)}\le\|wQ\|_K$.  Hence
\begin{equation*}
    \ol W_2(K_0,\rho_\infty)\le\ol W_\infty(K,\mu),
    \qquad
    S_2(K_0,w^2)=S_\infty(K,w).
\end{equation*}
The $p=2$ result gives \eqref{unconditional-pw-bound}.  Thus $(P)$ holds in
all cases, and Theorem~\ref{szego-implications-thm}(a) gives $(B)$.
\end{proof}

\subsection{The four-way theorem}

\begin{corollary}[Four-way Szeg\H{o} theorem] \label{four-way-thm}
Let $2\le p\le\infty$, and assume the hypotheses of
Theorem~\ref{szego-implications-thm}.  When $p=\infty$, assume also that $w$
is upper semicontinuous on $K$.  Then any three of the following four
assertions imply the remaining one:
\begin{enumerate}
\item[$(P)$] $K_0$ is a Parreau--Widom set.
\item[$(B)$] $X=\{x_j\}$ satisfies the Blaschke condition, $\BC(K)<\infty$.
\item[$(S)$] The weight $w$ satisfies the Szeg\H{o} condition, $S(K,w)>0$.
\item[$(W_p)$] $0<\limsup_{n\rightarrow\infty}W_{p,n}(K,\mu)<\infty$.
\end{enumerate}
Equivalently, if any two of these four assertions hold, then the remaining two assertions are equivalent.
\end{corollary}

\begin{proof}
If $(P)$, $(B)$, and $(S)$ hold, Corollary~\ref{3szego-thm}(a) gives
$(U_p)$ and $(L_p)$, and hence $(W_p)$.  If $(P)$, $(B)$, and $(W_p)$ hold
but $(S)$ fails, Corollary~\ref{3szego-thm}(b) gives
$W_{p,n}(K,\mu)\rightarrow0$, contradicting $(W_p)$.  If $(P)$, $(S)$, and
$(W_p)$ hold, then $(W_p)$ gives $(U_p)$ and
Theorem~\ref{szego-implications-thm}(a) gives $(B)$.  Finally, if $(B)$,
$(S)$, and $(W_p)$ hold, then $(W_p)$ gives $(U_p)$, and
Theorem~\ref{unconditional-pw-thm} gives $(P)$.
\end{proof}

We record the following specialization of the four-way Szeg\H{o} theorem. It shows that, on a regular non-Parreau--Widom compact subset of $\bb R$, neither the unweighted Chebyshev Widom factors nor the equilibrium-measure $L^2$ Widom factors can be bounded.

\begin{corollary}[Widom-factor characterization of Parreau--Widom sets] \label{regular-pw-characterization-cor}
Let $K\subset\bb R$ be a regular compact set of positive capacity. Then the following conditions are equivalent:
\begin{enumerate}[\quad$(a)$]
\item $K$ is a Parreau--Widom set.
\item The unweighted Chebyshev Widom factors are bounded:
\begin{equation*}
    \sup_{n\ge1}W_{\infty,n}(K,1)<\infty.
\end{equation*}
\item The equilibrium-measure $L^2$ Widom factors are bounded:
\begin{equation*}
    \sup_{n\ge1}W_{2,n}(K,\mu_K)<\infty.
\end{equation*}
\end{enumerate}
In particular, if $K$ has zero Lebesgue measure, then
\begin{equation*}
    \sup_{n\ge1}W_{\infty,n}(K,1)=\infty,
    \qquad
    \sup_{n\ge1}W_{2,n}(K,\mu_K)=\infty.
\end{equation*}
\end{corollary}
\begin{proof}
Take $K_0=K$, $X=\varnothing$, $w\equiv1$, and $d\mu=d\mu_K$.  Then
$(B)$ and $(S)$ hold.  The equivalent formulation of
Corollary~\ref{four-way-thm} therefore gives
$(P)\Longleftrightarrow(W_p)$ for $p\in\{2,\infty\}$.  Moreover,
Theorem~\ref{szego-implications-thm}(a) gives $(L_p)$, so $(W_p)$ is
equivalent to $(U_p)$ and
$(U_p)$ is equivalent to boundedness of the corresponding sequence.
The final assertion follows from
\cite[Proposition, Appendix~A]{Chr12}.
\end{proof}

\section{Examples and separation of the conditions}

We use the order $(P,B,S,U_p,L_p)$ throughout this section.  A plus sign
means that the corresponding condition holds, and a minus sign means that it
fails.  The range of $p$ is stated whenever an example is not valid for all
$0<p\le\infty$.  In the elementary examples that include $p=\infty$, if only
a measure is displayed, we choose the independent Chebyshev weight to equal
its density on $K_0$ and its atom masses on $X$.  These weights satisfy the
corresponding hypotheses in Theorem~\ref{szego-implications-thm}.

The interval equilibrium potential in \cite[Example~I.3.5]{ST97}, together
with \eqref{GrFn}, gives, for $a<b$,
\begin{equation}\label{interval-endpoint-green}
    \lim_{t\downarrow0}\frac{g_{[a,b]}(b+t)}{\sqrt t}
    =\lim_{t\downarrow0}\frac{g_{[a,b]}(a-t)}{\sqrt t}
    =\frac{2}{\sqrt{b-a}}.
\end{equation}

\subsection{Elementary examples}

We first record a source of regular sets that are not Parreau--Widom and will be used
several times.

\begin{lemma}
\label{non-pw-lem}
There is a regular compact set $K_0\subset\bb R$ of positive capacity for
which $\PW(K_0)=\infty$.
\end{lemma}

\begin{proof}
Choose $0<\eta_1<\eta_2<\cdots\uparrow\eta_*<1$.  In the comb
representation from \cite[Section~4, pp.~108--112]{EreYud12}, take vertical
slits based at $\pi\eta_j$ with heights $h_j=1/j$, and put the height equal to
zero at all other points of the base.  Under the inverse construction, these
slits correspond to bounded gaps with critical Green values $h_j$.  Since,
for every $\eps>0$, only finitely many slit heights exceed $\eps$, the
regularity criterion in the same section applies.  The resulting compact set
$K_0$ has positive capacity and $\PW(K_0)=\sum_jh_j=\infty$.
\end{proof}

\begin{example}\label{elementary-patterns-ex}
The following choices realize five sign patterns.
\begin{enumerate}[(a)]
\item Let $K=[-1,1]$ and $d\mu=d\mu_K$.  Then
$(P,B,S,U_p,L_p)=(+,+,+,+,+)$ for every $0<p\le\infty$.
Indeed, $(P)$, $(B)$, and $(S)$ are immediate, while both Widom-factor
conditions follow from Corollary~\ref{3szego-thm}.

\item Let $K=[-1,1]$ and
$w(x)=\exp(-(1-x^2)^{-1})$ for $-1<x<1$, with $w(\pm1)=0$.  Since
$\int\log w\,d\mu_K=-\infty$, we have $S(K,w)=0$.  For finite $p$, take
$d\mu=w\,d\mu_K$.  For $p=\infty$, use $w$ as the weight.  Theorem~\ref{szego-implications-thm}(b) shows that the corresponding Widom factors tend to zero.  Thus $(P,B,S,U_p,L_p)=(+,+,-,+,-)$ for every $0<p\le\infty$.

\item Let $K_0=[-1,1]$, $X=\{1+1/j:j\ge1\}$, and
$K=K_0\cup X$.  Set
$d\mu=d\mu_K+\sum_{j\ge1}2^{-j}d\delta_{1+1/j}$.  Here $S(K,w)=1$.
Since the polar set $X$ does not change the Green function, the
endpoint limit \eqref{interval-endpoint-green} gives
$g_K(1+1/j)\ge j^{-1/2}$ for all sufficiently large $j$.  Thus
$\BC(K)=\infty$.  The first part of
Theorem~\ref{szego-implications-thm}
therefore gives $(P,B,S,U_p,L_p)=(+,-,+,-,+)$ for every
$0<p\le\infty$.

\item Let $K=K_0$ be supplied by Lemma~\ref{non-pw-lem} and take
$d\mu=d\mu_K$.  Then $(B)$ and $(S)$ hold.  The lower bound
\eqref{main-lower-eq} in Theorem~\ref{main-estimates-thm} gives $(L_p)$,
whereas $(U_p)$ would contradict Theorem~\ref{unconditional-pw-thm}.  Hence
$(P,B,S,U_p,L_p)=(-,+,+,-,+)$ for every $2\le p\le\infty$.

\item Let $K_0$ be as in part (d), and choose isolated points
$X=\{x_j\}\subset\bb R\bs K_0$ accumulating only on $K_0$ so that
$\sum_jg_{K_0}(x_j)=\infty$.  For instance, if $b=\max K_0$, choose
$x_j\downarrow b$ with $g_{K_0}(x_j)=1/j$.  Put $K=K_0\cup X$.  For
$d\mu=d\mu_{K_0}+\sum_{j\ge1}2^{-j}d\delta_{x_j}$, we have
$S(K,w)=1$ and $\BC(K)=\infty$.  Hence the lower bound
\eqref{main-lower-eq} in Theorem~\ref{main-estimates-thm} gives
$\liminf_nW_{p,n}(K,\mu)=\infty$.  Thus
$(P,B,S,U_p,L_p)=(-,-,+,-,+)$ for every $0<p\le\infty$.
\end{enumerate}
\end{example}

\subsection{Discrete measures}

A finite positive measure $\mu$ is called discrete if
$\mu=\sum_j a_j\delta_{x_j}$, where $a_j>0$ and $\sum_j a_j<\infty$.

We shall use the standard fact that finitely supported probability measures on
a compact set are weak-star dense among its probability measures.  See
\cite[Example~8.1.6(i), pp.~176--177]{Bog07}.

The next proposition supplies two general constructions for $1<p<\infty$.  Part~(a) shows that every support in the class carries a discrete measure for which $(S)$ fails and the Widom factors tend to zero.
Part~(b) starts with a support carrying some measure with unbounded Widom
factors and produces a discrete measure on the same support for which
$(S)$ fails and
\begin{equation*}
    \liminf_{n\rightarrow\infty}W_{p,n}=0,
    \qquad
    \limsup_{n\rightarrow\infty}W_{p,n}=\infty.
\end{equation*}
These constructions provide the zero and oscillatory behaviors used in
the separation examples below.  Part~(a) is the finite-$p$ version of
\cite[Theorem~5.5 and Corollary~5.6]{Sim07}.

\begin{proposition}\label{oscillating-singularization}
Let $1<p<\infty$.  Let $K_0\subset\bb R$ be a regular compact set of positive
capacity, and let $X\subset\bb R\bs K_0$ be finite or countable.  Assume that
$K=K_0\cup X$ is compact and that every point of $X$ is isolated in $K$.
\begin{enumerate}[(a)]
\item There is a finite positive discrete measure $\mu$ such that
$\supp(d\mu)=K$ and $W_{p,n}(K,\mu)\rightarrow0$.

\item If a finite positive measure $\nu$ satisfies $\supp(d\nu)=K$ and
$\sup_nW_{p,n}(K,\nu)=\infty$, then there is a finite positive discrete
measure $\mu$ such that $\supp(d\mu)=K$ and
\begin{equation}\label{discrete-oscillation-eq}
    \liminf_{n\rightarrow\infty}W_{p,n}(K,\mu)=0,
    \qquad
    \limsup_{n\rightarrow\infty}W_{p,n}(K,\mu)=\infty.
\end{equation}
\end{enumerate}
In both parts, $d\mu\perp d\mu_K$, so the density $w$ in
\eqref{mu-main-decomp} vanishes $\mu_K$-almost everywhere and $(S)$ fails.
\end{proposition}

\begin{proof}
For both parts, fix $D>\operatorname{diam}(K)$.  Then
$|x-y|<D$ for $x,y\in K$, while
$\ca(K)\le\operatorname{diam}(K)/2$ by
\cite[Theorem~5.3.4, p.~140]{Ran95}.  In particular,
$0<\ca(K)/(2D)<1$.
Choose a dense sequence $\{z_m\}$ in $K$ containing $X$.  For part (a), set
$r=(\ca(K)/(2D))^p\in(0,1)$ and $a_m=(1-r)r^{m-1}$.  Then
$\sum_{m>n}a_m=r^n=(2^{-n}\ca(K)^n/D^n)^p$ for every $n\ge1$.  Set
$d\mu=\sum_{m\ge1}a_m\,d\delta_{z_m}$.  Since
$\sum_{m\ge1}a_m=1$, this is a probability measure.  Moreover,
$\supp(d\mu)=K$ because $a_m>0$ for every $m$ and $\{z_m\}$ is dense in
$K$.  The monic polynomial $P_n(x)=\prod_{m=1}^n(x-z_m)$ satisfies
$\|P_n\|_K\le D^n$ and
vanishes at $z_1,\ldots,z_n$.  Therefore
\begin{align*}
    t_{p,n}(K,\mu)^p
    &\le \int_K|P_n|^p\,d\mu
    =\sum_{m>n}a_m|P_n(z_m)|^p\\
    &\le D^{np}\sum_{m>n}a_m
    =D^{np}r^n
    =2^{-np}\ca(K)^{np}.
\end{align*}
Taking $p$th roots and dividing by $\ca(K)^n$ gives
$W_{p,n}(K,\mu)\le2^{-n}$.  This proves part~(a).

For part (b), normalize $\nu$ to be a probability measure.  This only
multiplies all its Widom factors by the same positive constant, so they remain
unbounded.  We inductively choose positive numbers $\tau_k$, finitely supported
probability measures $\nu_k$, and integers $m_k,N_k$.  Set $\tau_1=1$.
Once $\tau_k$ has been defined, the unboundedness of
$\{W_{p,n}(K,\nu)\}$ allows us to choose $m_k>k$ such that
$W_{p,m_k}(K,\nu)>4k\tau_k^{-1/p}$.  Since
$W_{p,m_k}(K,\nu)=t_{p,m_k}(K,\nu)/\ca(K)^{m_k}$, this is equivalent to
$t_{p,m_k}(K,\nu)>4k\tau_k^{-1/p}\ca(K)^{m_k}$.

Since $m_k$ is fixed, \cite[Theorem~2.1]{AZ20b} allows us to choose a finitely
supported probability measure $d\nu_k$ on $K$ such that
\begin{equation}\label{osc-large-block}
    t_{p,m_k}(K,\nu_k)>2k\tau_k^{-1/p}\ca(K)^{m_k}.
\end{equation}
In particular, $t_{p,m_k}(K,\nu_k)>0$, and hence
$|\supp(d\nu_k)|\ge m_k+1$.  Indeed, otherwise a monic polynomial of degree
$m_k$ could vanish on $\supp(d\nu_k)$ and would have zero $L^p(d\nu_k)$ norm.
Let
$A_k=\supp(d\nu_1)\cup\cdots\cup\supp(d\nu_k)
\cup\{z_1,\ldots,z_k\}$,
set $N_k=|A_k|+1$.  Since $\supp(d\nu_k)\subset A_k$, we have
$N_k\ge m_k+2>k$.  Finally, define
\begin{equation}\label{osc-parameter-def}
    \varepsilon_k=
        \left(\frac{2^{-k}\ca(K)^{N_k}}{D^{N_k}}\right)^p,
    \qquad
    \tau_{k+1}=2^{-k-2}\varepsilon_k.
\end{equation}
This completes the $k$th step.  Moreover,
$A_{k-1}\subset A_k$, so $N_{k-1}\le N_k$.  Consequently, for $k>1$,
\begin{equation}\label{osc-epsilon-decrease}
    \frac{\varepsilon_k}{\varepsilon_{k-1}}
    =\left(\frac12
      \left(\frac{\ca(K)}{D}\right)^{N_k-N_{k-1}}\right)^p
    \le 2^{-p}<1.
\end{equation}

Now set
$d\mu=\sum_{k\ge1}\tau_k(d\nu_k+d\delta_{z_k})$.
By \eqref{osc-epsilon-decrease}, $\varepsilon_k\le\varepsilon_1$, and hence
\eqref{osc-parameter-def} gives
$\sum_k\tau_k\le1+\varepsilon_1\sum_{k\ge2}2^{-k-1}
=1+\varepsilon_1/4<\infty$.
Thus $\mu$ is a finite discrete measure with $\supp(d\mu)=K$.  The
degrees $m_k$ give the large Widom factors.  Indeed,
$d\mu\ge\tau_kd\nu_k$, so
\eqref{osc-large-block} gives
\begin{equation}\label{osc-unbounded-W-subsequence}
    W_{p,m_k}(K,\mu)
    \ge \frac{\tau_k^{1/p}t_{p,m_k}(K,\nu_k)}{\ca(K)^{m_k}}>2k.
\end{equation}

The degrees $N_k$ give the small Widom factors.  Choose $y_k\in K$ and set
$Q_k(x)=(x-y_k)\prod_{a\in A_k}(x-a)$.  This polynomial is monic of degree
$N_k$, vanishes on $A_k$, and has all its zeros in $K$.  Hence
\begin{equation}\label{osc-Qk-norm-eq}
    \|Q_k\|_K\le D^{N_k}.
\end{equation}
If $\ell\le k$, then
$\supp(d\nu_\ell)\cup\{z_\ell\}\subset A_k$, and therefore
$\int_K|Q_k|^p\,d\nu_\ell+|Q_k(z_\ell)|^p=0$.  If $\ell>k$, then
$\varepsilon_{\ell-1}\le\varepsilon_k$ by
\eqref{osc-epsilon-decrease}, while \eqref{osc-parameter-def} gives
$\tau_\ell=2^{-\ell-1}\varepsilon_{\ell-1}
\le2^{-\ell-1}\varepsilon_k$.  Therefore
\begin{equation}\label{osc-tail-at-Nk}
    2\sum_{\ell>k}\tau_\ell
    \le 2\varepsilon_k\sum_{\ell>k}2^{-\ell-1}
    =2^{-k}\varepsilon_k<\varepsilon_k.
\end{equation}
Since every $d\nu_\ell$ is a probability measure,
\eqref{osc-Qk-norm-eq} and \eqref{osc-tail-at-Nk} give
\begin{align*}
    t_{p,N_k}(K,\mu)^p
    &\le \int_K|Q_k|^p\,d\mu\\
    &=\sum_{\ell>k}\tau_\ell
      \left(\int_K|Q_k|^p\,d\nu_\ell+|Q_k(z_\ell)|^p\right)\\
    &\le 2D^{pN_k}\sum_{\ell>k}\tau_\ell
    <D^{pN_k}\varepsilon_k
    =2^{-kp}\ca(K)^{pN_k}.
\end{align*}
Taking $p$th roots and dividing by $\ca(K)^{N_k}$ gives
\begin{equation}\label{osc-zero-W-subsequence}
    0\le W_{p,N_k}(K,\mu)\le2^{-k}.
\end{equation}
Since $m_k>k$ and $N_k\ge m_k+2$, both sequences of degrees tend to infinity.
Hence \eqref{osc-unbounded-W-subsequence} and
\eqref{osc-zero-W-subsequence} give
$\limsup_{n\rightarrow\infty}W_{p,n}(K,\mu)=\infty$ and
$\liminf_{n\rightarrow\infty}W_{p,n}(K,\mu)=0$, respectively.  This proves
\eqref{discrete-oscillation-eq}.  Finally, both measures constructed above are
concentrated on countable sets, whereas $\mu_K$ has no point masses.  Thus they
are singular with respect to $\mu_K$, and their density $w$ in
\eqref{mu-main-decomp} vanishes
$\mu_K$-almost everywhere.  Hence $S(K,w)=0$.
\end{proof}

\begin{corollary}\label{discrete-patterns-cor}
For every $1<p<\infty$, the patterns in rows
$\mathrm{(a)}$--$\mathrm{(d)}$ and $\mathrm{(f)}$ below are realized by
discrete measures $d\mu$, where $K=\supp(d\mu)$.  Row $\mathrm{(e)}$ is
realized for every $2\le p<\infty$.
\begin{equation*}
\begin{array}{c|ccc}
 & (P,B,S,U_p,L_p) & & (P,B,S,U_p,L_p)\\
\hline
(a)&(+,-,-,+,-)&(d)&(+,-,-,-,-)\\
(b)&(-,+,-,+,-)&(e)&(-,+,-,-,-)\\
(c)&(-,-,-,+,-)&(f)&(-,-,-,-,-).
\end{array}
\end{equation*}
\end{corollary}

\begin{proof}
Let $K_0$ be supplied by Lemma~\ref{non-pw-lem}.  Consider the three compact
sets
$K^{(1)}=[-1,1]\cup\{1+1/j:j\ge1\}$, $K^{(2)}=K_0$, and
$K^{(3)}=K_0\cup X$, where $X$ is chosen so that $\BC(K^{(3)})=\infty$.
Part (a) of Proposition~\ref{oscillating-singularization}, applied to these
three sets, gives rows (a)--(c).

For rows (d)--(f), use respectively the model measures
$d\mu_{[-1,1]}+\sum_{j\ge1}2^{-j}d\delta_{1+1/j}$, $d\mu_{K_0}$, and
$d\mu_{K_0}+\sum_{x_j\in X}2^{-j}d\delta_{x_j}$.  The first and third have
unbounded Widom factors by the lower estimate in
Theorem~\ref{main-estimates-thm}.  The second does as well, since otherwise
Theorem~\ref{unconditional-pw-thm} would make $K_0$ Parreau--Widom.  Part (b) of
Proposition~\ref{oscillating-singularization} produces discrete measures
whose supports are still $K^{(1)}$, $K^{(2)}$, and $K^{(3)}$, respectively.
Since $(B)$ depends only on the support through \eqref{PW-BC-def}, the
constructed measures have the same Blaschke signs $(-,+,-)$.  This gives
rows (d)--(f).
\end{proof}

\begin{example}\label{ex-B-L2-only}
There is a regular compact set $K\subset\bb R$ of zero Lebesgue measure.  It
supports a discrete probability measure $\mu$ such that
$(P,B,S,U_2,L_2)=(-,+,-,-,+)$.

Let $K=K(\gamma)$ be the weakly equilibrium Cantor set of
\cite[Section~2]{AlpGon16} corresponding to $\gamma_s=2^{-3s-2}$, $s\ge1$.
Then $\gamma_s\le\gamma_1=1/32$.  The capacity formula and the zero-measure
estimate in \cite[p.~3783]{AlpGon16} give
$\ca(K)=\exp(\sum_{s\ge1}2^{-s}\log\gamma_s)=2^{-8}>0$ and $|K|=0$.
Since $\ca(K)>0$, the capacity formula gives
$\sum_{s\ge1}2^{-s}\log(1/\gamma_s)<\infty$.  Put
$\delta_s=\gamma_1\cdots\gamma_s$ and
$\rho_s=\sum_{k=s+1}^\infty2^{-k}\log(1/(2\gamma_k))$.  If
$\delta_s\le\delta<\delta_{s-1}$, \cite[Theorem~5]{Gon14} gives
\begin{equation*}
    \sup_{\dist(z,K)\le\delta}g_K(z)
    <\rho_s+2^{-s}\log\frac{16\delta}{\delta_s}
    <\rho_s+2^{-s}\log\frac{16}{\gamma_s}.
\end{equation*}
The convergence above implies
$\rho_s\rightarrow0$ and $2^{-s}\log(16/\gamma_s)\rightarrow0$.  Therefore
$g_K(z)\rightarrow0$ as $z$ approaches $K$, and $K$ is regular by
\cite[Theorem~4.4.9]{Ran95}.  If
$2^s\le n<2^{s+1}$, where $s\ge0$, the proof of
\cite[Theorem~5.1(a)]{AlpGon16} and \cite[Example~5.2]{AlpGon16} give
$W_{2,n}(K,\mu_K)\ge1/\sqrt{6\gamma_{s+1}}$.  Since
$\gamma_{s+1}^{-1}=2^{3s+5}\ge4(n+1)^3$, it follows that
\begin{equation}\label{discrete-example-equilibrium-bound}
    W_{2,n}(K,\mu_K)\ge\frac12(n+1)^{3/2},
    \quad n\ge1.
\end{equation}

Choose a sequence $\{y_n\}$ dense in $K$, and let
$\{\nu_j\}_{j\ge2}$ be finitely supported probability measures on $K$ such
that $\nu_j\rightarrow\mu_K$ in the weak-star topology.  For every fixed
$n$, $(1-j^{-1})\nu_j+j^{-1}\delta_{y_n}\rightarrow\mu_K$ in the weak-star
topology as $j\rightarrow\infty$.  Hence, by the first assertion of
\cite[Theorem~2.1]{AZ20b}, we may choose $j_n\ge2$ so that the finitely
supported probability measure
$d\sigma_n=(1-j_n^{-1})d\nu_{j_n}+j_n^{-1}d\delta_{y_n}$ satisfies
\begin{equation}\label{discrete-example-approximation}
    t_{2,n}(K,\sigma_n)^2
    \ge\frac12t_{2,n}(K,\mu_K)^2.
\end{equation}
Moreover, $\sigma_n(\{y_n\})\ge j_n^{-1}>0$.

Let $A=\sum_{n\ge1}(n+1)^{-2}$ and define
\begin{equation}\label{discrete-example-measure}
    d\mu=A^{-1}\sum_{n\ge1}(n+1)^{-2}d\sigma_n.
\end{equation}
Since every $\sigma_n$ is a probability measure, so is $\mu$.  It is
discrete, and $\mu(\{y_n\})>0$ for every $n$.  Since $\{y_n\}$ is dense in $K$
and every $\sigma_n$ is supported on $K$, we have $\supp(d\mu)=K$.  From
\eqref{discrete-example-measure},
$d\mu\ge A^{-1}(n+1)^{-2}d\sigma_n$.  Hence
\eqref{discrete-example-approximation} and
\eqref{discrete-example-equilibrium-bound} give
\begin{align*}
    W_{2,n}(K,\mu)^2
    &\ge \frac{(n+1)^{-2}}{A}
    \frac{t_{2,n}(K,\sigma_n)^2}{\ca(K)^{2n}}\\
    &\ge \frac{(n+1)^{-2}}{2A}W_{2,n}(K,\mu_K)^2\\
    &\ge \frac{n+1}{8A}.
\end{align*}
Thus $W_{2,n}(K,\mu)\rightarrow\infty$, so $(U_2)$ fails and $(L_2)$ holds.
Since $X=\varnothing$, $(B)$ holds.  The measure $\mu$ is concentrated on a
countable set, whereas $\mu_K$ has no point masses.  Thus
$d\mu\perp d\mu_K$ and $(S)$ fails.
Finally, a regular zero-measure real compact set is not Parreau--Widom
\cite[Proposition, Appendix~A]{Chr12}.  Hence
$(P,B,S,U_2,L_2)=(-,+,-,-,+)$.
\end{example}

\begin{example}\label{ex-only-L2}
Let $K_0$ denote the set $K$ from Example~\ref{ex-B-L2-only}, and let $d\mu$
be the measure constructed there.  Enumerate the
bounded gaps of $K_0$ as $(\alpha_j,\beta_j)$, let $c_j$ be the critical point
of $g_{K_0}$ in the $j$th gap, and set $X=\{c_j:j\ge1\}$ and $K=K_0\cup X$.
Every point of $X$ is isolated in $K$, and every accumulation point of $X$
belongs to $K_0$, so $K$ is compact and $K_{\rm reg}=K_0$.  Since $X$ is
polar, $\ca(K)=\ca(K_0)$ and $\mu_K=\mu_{K_0}$.  Moreover,
\begin{equation*}
    \BC(K)=\sum_{j\ge1}g_{K_0}(c_j)=\PW(K_0)=\infty.
\end{equation*}

Define
\begin{equation*}
    d\widehat\mu=\frac12d\mu+
    \frac12\sum_{j\ge1}2^{-j}d\delta_{c_j}.
\end{equation*}
Then $\widehat\mu$ is a discrete probability measure with
$\supp(d\widehat\mu)=K$.  Since $\mu_K=\mu_{K_0}$ has no point masses, $(S)$
fails.  Since
$d\widehat\mu\ge\frac12d\mu$ and $\ca(K)=\ca(K_0)$,
\begin{equation*}
    W_{2,n}(K,\widehat\mu)^2
    \ge\frac12W_{2,n}(K_0,\mu)^2\rightarrow\infty.
\end{equation*}
Thus $(P,B,S,U_2,L_2)=(-,-,-,-,+)$.
\end{example}

\subsection{Jacobi and rank-one constructions}

For a self-adjoint operator $A$, write $\sigma(A)$ and
$\sigma_{\rm ess}(A)$ for its spectrum and essential spectrum.  The inner
product $\langle u,v\rangle$ is linear in $u$.

If $d\nu$ is the probability spectral measure of a half-line Jacobi matrix
with off-diagonal coefficients $\{a_n\}_{n\ge1}$, write
$K=\supp(d\nu)$.  Then $t_{2,n}(K,\nu)=a_1\cdots a_n$.  See
\cite[Theorems~1.2.4--1.2.5]{Sim11}.

\begin{example}\label{jacobi-patterns-ex}
Let $d\mu$ be the spectral measure of a half-line Jacobi matrix $J$ for
$\delta_1$.  In both examples below, $a_n\rightarrow1$ and $b_n\rightarrow0$.
Thus the essential spectrum of $J$ is $[-2,2]$, and the eigenvalues outside
this interval are isolated and can accumulate only at $-2$ and $2$.  See
\cite[Theorem~1.4.1 and Remark~1, p.~18]{Sim11}.  Consequently,
$K=\supp(d\mu)=[-2,2]\cup X$, where $X$ is the exterior point spectrum,
$K_0=[-2,2]$ is Parreau--Widom, and $\ca(K)=1$.
\begin{enumerate}[(a)]
\item Take $a_n=1$ and $b_n=n^{-1/4}$.  Since $b\notin\ell^4$ and
$\{b_{n+1}-b_n\}\in\ell^2$, \cite[Corollary~2]{Zla05} gives
$\sum_{x\in X}(|x|-2)^{5/2}=\infty$.  By
\eqref{interval-endpoint-green},
$g_{[-2,2]}(x)\ge(|x|-2)^{5/2}$ for every $x\in X$ sufficiently close to
$-2$ or $2$.  All but finitely many points of $X$ have this property, so the
Blaschke sum diverges.
Moreover, $t_{2,n}=a_1\cdots a_n=1$, so $W_{2,n}=1$.  If $(S)$ held,
the implication $(S)+(U_2)\Rightarrow(B)$ in
Theorem~\ref{szego-implications-thm} would give $(B)$.  Therefore
$(P,B,S,U_2,L_2)=(+,-,-,+,+)$.

\item Take $a_n=1+1/n$ and $b_n=3/n$.  Then
$W_{2,n}=a_1\cdots a_n=n+1$, so $(U_2)$ fails and $(L_2)$ holds.  Here
$n(a_n-1)\rightarrow1$ and $nb_n\rightarrow3$, and
\cite[Theorem~3]{SZ03} shows that $d\mu$ does not satisfy the Szeg\H{o}
condition on $[-2,2]$, so $(S)$ fails.  Finally, $(B)$ would imply
$(U_2)$ by the implication $(P)+(B)\Rightarrow(U_2)$ in
Theorem~\ref{szego-implications-thm}.  Hence
$(P,B,S,U_2,L_2)=(+,-,-,-,+)$.
\end{enumerate}
\end{example}

\subsection{Zero-measure Schr\"odinger constructions}

\begin{proposition}
\label{zero-measure-schrodinger-prop}
Let $J_+$ be a bounded half-line discrete Schr\"odinger operator whose essential
spectrum $K_0$ is regular, has capacity one, and has Lebesgue measure zero.  For
$t\in\bb R$, let $d\mu_t$ be the spectral measure of
$J_+(t)=J_++t\langle\,\cdot\,,\delta_1\rangle\delta_1$, put
$K_t=\supp(d\mu_t)$, and set $X_t=K_t\bs K_0$.  Then $K_t=K_0\cup X_t$ and
$W_{2,n}(K_t,\mu_t)=1$ for every $n\ge1$.  For Lebesgue-a.e.
$t$, the Szeg\H{o} condition fails, and the pattern is
$(-,+,-,+,+)$ if $X_t$ is Blaschke and $(-,-,-,+,+)$ otherwise.
\end{proposition}

\begin{proof}
Every off-diagonal Jacobi coefficient is one, so $t_{2,n}(K_t,\mu_t)=1$.
The vector $\delta_1$ is cyclic, and hence $K_t=\sigma(J_+(t))$.  See
\cite[Section~2.2.3 and Theorem~1.6.6]{DamFilI}.  Since a rank-one
perturbation does not change the essential spectrum
\cite[Corollary~1.4.25]{DamFilI}, the points of $X_t$ are isolated
eigenvalues.  Thus $X_t$ is countable and polar, and
$\ca(K_t)=\ca(K_0)=1$ and $\mu_{K_t}=\mu_{K_0}$.
Moreover, $K_0$ is not Parreau--Widom because it has zero Lebesgue measure
\cite[Proposition, Appendix~A]{Chr12}.  By spectral averaging
\cite[Theorem~1.1 and its proof on p.~287]{Mar11}, $\mu_t(K_0)=0$ for
Lebesgue-a.e. $t$.  For these values of $t$,
$d\mu_t\perp d\mu_{K_0}=d\mu_{K_t}$, so $(S)$ fails.
The two alternatives are then determined by
$\sum_{x\in X_t}g_{K_0}(x)$.
\end{proof}

For $\lambda>0$, $\alpha\in\bb R$ with $\alpha/\pi\notin\bb Q$, and
$\theta\in\bb R/(2\pi\bb Z)$, let $H_{\lambda,\alpha,\theta}$ be the
whole-line almost Mathieu operator
\begin{equation*}
    (H_{\lambda,\alpha,\theta}u)_n
    =u_{n+1}+u_{n-1}
    +\lambda\cos(n\alpha+\theta)u_n,
    \quad n\in\bb Z.
\end{equation*}
The parameters $\lambda$, $\alpha$, and $\theta$ are the coupling, frequency,
and phase, respectively.  The critical coupling is $\lambda=2$.
Let $H^+_{\lambda,\alpha,\theta}$ denote the Dirichlet
half-line restriction to $\ell^2(\bb N)$, given by the same formula for
$n\ge1$ with $u_0=0$.

By \cite[Theorem~7.1(a), p.~741 and the almost-periodic case on
p.~743]{Sim07} and \cite[Example~8.3, p.~746]{Sim07}, there is a compact set
$K_{\lambda,\alpha}$,
independent of $\theta$, such that
\begin{equation}\label{amo-half-line-essential-spectrum}
    \sigma(H_{\lambda,\alpha,\theta})
    =\sigma_{\rm ess}(H^+_{\lambda,\alpha,\theta})
    =K_{\lambda,\alpha}.
\end{equation}
At the critical coupling $\lambda=2$, the results in
\cite[Example~8.3, equations~(7.17) and~(8.10), and
Theorem~8.4]{Sim07} give
\begin{equation}\label{critical-amo-properties}
    |K_{2,\alpha}|=0,\qquad
    \ca(K_{2,\alpha})=1,\qquad
    K_{2,\alpha}\text{ is regular}.
\end{equation}
Every $H^+_{2,\alpha,\theta}$ therefore has the hypotheses of
Proposition~\ref{zero-measure-schrodinger-prop}.

\begin{example}\label{zero-measure-fixed-ex}
Fix $\alpha\in\bb R$ with $\alpha/\pi\notin\bb Q$, and set
$K_0=K_{2,\alpha}$.  We realize the patterns $(-,+,-,+,+)$ and
$(-,-,-,+,+)$ with measures whose supports have regular part $K_0$.  By
\eqref{critical-amo-properties}, $K_0$ is regular, has capacity one, and has
Lebesgue measure zero.  In particular, it is not Parreau--Widom
\cite[Proposition, Appendix~A]{Chr12}.
\begin{enumerate}[(a)]
\item Put $H=H_{2,\alpha,-\alpha/2}$ and
$J_\pm=H^+_{2,\alpha,-\alpha/2}
\mathbin{\pm}\langle\,\cdot\,,\delta_1\rangle\delta_1$.
The potential of $H$ is $b_n=2\cos\bigl((n-\tfrac12)\alpha\bigr)$, so
$b_{1-n}=b_n$.  Define $R$ on $\ell^2(\bb Z)$ by $(Ru)_n=u_{1-n}$.  Since
$\|Ru\|^2=\sum_{n\in\bb Z}|u_{1-n}|^2=\|u\|^2$ and $R^2=I$, the operator
$R$ is a surjective isometry and hence unitary.  Moreover,
$R^*=R^{-1}=R$.  Put
$\mathcal E=\ker(R-I)$ and $\mathcal O=\ker(R+I)$.  Every
$u\in\ell^2(\bb Z)$ has the decomposition
$u=(u+Ru)/2+(u-Ru)/2$, whose two terms belong to $\mathcal E$ and
$\mathcal O$, respectively.  These two subspaces are orthogonal and
in particular, the decomposition is unique.  Thus
$\ell^2(\bb Z)=\mathcal E\oplus\mathcal O$.  The identity
$b_{1-n}=b_n$ gives $RH=HR$, so both subspaces reduce $H$.

Define $U_+:\mathcal E\rightarrow\ell^2(\bb N)$ and
$U_-:\mathcal O\rightarrow\ell^2(\bb N)$ by
$(U_\pm u)_n=\sqrt2u_n$ for $n\ge1$.  If $u\in\mathcal E$ or
$u\in\mathcal O$, respectively, then $u_{1-n}=\pm u_n$, and hence
$\|u\|^2=2\sum_{n\ge1}|u_n|^2=\|U_\pm u\|^2$.  Conversely, the inverse is
obtained by setting $u_n=v_n/\sqrt2$ and
$u_{1-n}=\pm v_n/\sqrt2$ for $n\ge1$.  Thus $U_+$ and $U_-$ are unitary.
Fix $s\in\{1,-1\}$.  When $s=1$, take $u\in\mathcal E$, $U_s=U_+$, and
$J_s=J_+$.  When $s=-1$, take $u\in\mathcal O$, $U_s=U_-$, and $J_s=J_-$.
Put $v=U_su$.  Since the chosen subspace reduces $H$, the vector $Hu$
belongs to the domain of $U_s$.  Taking $n=1$ in $u_{1-n}=su_n$ gives
$u_0=su_1$.  Since $(Hu)_1=u_2+u_0+b_1u_1$, it follows that
$(U_sHu)_1=v_2+(b_1+s)v_1$.  For $n\ge2$, both $n-1$ and $n+1$ belong to
$\bb N$, and hence $(U_sHu)_n=v_{n+1}+v_{n-1}+b_nv_n$.  By the definition
of $J_s$, we also have $(J_sv)_1=v_2+(b_1+s)v_1$ and
$(J_sv)_n=v_{n+1}+v_{n-1}+b_nv_n$ for $n\ge2$.  Therefore
$U_sHu=J_sv$.  Consequently, $U=U_+\oplus U_-$ is a unitary map from
$\ell^2(\bb Z)=\mathcal E\oplus\mathcal O$ onto
$\ell^2(\bb N)\oplus\ell^2(\bb N)$, and
\begin{equation*}
    UHU^{-1}=J_+\oplus J_-.
\end{equation*}
The same decomposition is recorded in
\cite[Section~7, (7.1), p.~238]{Zla04Sparse}.  By
\eqref{amo-half-line-essential-spectrum} and invariance of the essential
spectrum under finite-rank perturbations
\cite[Corollary~1.4.25]{DamFilI}, both $J_+$ and $J_-$ have essential
spectrum $K_0$.  Since $H$ has no eigenvalues \cite[Theorem~1.1]{Jit21}, neither
$J_+$ nor $J_-$ has an eigenvalue: otherwise an eigenvector of one summand,
extended by zero in the other, would be an eigenvector of
$J_+\oplus J_-\simeq H$.  If $d\mu_+$ and $d\mu_-$ are
their spectral
measures for $\delta_1$, Proposition~\ref{zero-measure-schrodinger-prop}
therefore gives $\supp(d\mu_+)=\supp(d\mu_-)=K_0$ and
$W_{2,n}(K_0,\mu_+)=W_{2,n}(K_0,\mu_-)=1$.

Both measures are singular with respect to Lebesgue measure.  Since the
off-diagonal coefficients are one, $J_+$ and $J_-$ correspond to the boundary
parameters $\beta=-1$ and $\beta=1$, respectively, in the notation of
\cite[(1.91) and Remark~1.9, p.~16]{Tes00}.  The singular parts for distinct
boundary parameters are mutually singular by
\cite[Lemma~3.11 and its proof, pp.~65--66]{Tes00}.  Therefore
$d\mu_+\perp d\mu_-$.  If both measures
satisfied $(S)$, then $d\mu_{K_0}\ll d\mu_+$ and
$d\mu_{K_0}\ll d\mu_-$, which is impossible.  Choose $s\in\{+,-\}$ so that
$d\mu_s$ fails $(S)$, and put
$J_0=J_s$ and $d\mu_0=d\mu_s$.  Here $X$ is empty, so $(B)$ holds.  Thus
$d\mu_0$ realizes $(P,B,S,U_2,L_2)=(-,+,-,+,+)$.

\item Keep $J_0$ from part~(a), and put
$\beta_0=\max K_0$ and $\varepsilon_j=j^{-2}$.  For $M\ge1$, let
$J_0^{(M)}$ be the Dirichlet restriction of $J_0$ to
$\ell^2(\{M,M+1,\ldots\})$.  By \cite[Lemma~3.7, p.~63]{Tes00},
$\sigma_{\rm ess}(J_0^{(M)})=K_0$, and hence
\cite[Proposition~5.11, (5.4), p.~134]{Luk22} gives
\begin{equation*}
    \sup_{\|u\|=1}\langle J_0^{(M)}u,u\rangle
    =\max\sigma(J_0^{(M)})\ge\beta_0.
\end{equation*}
Since finitely supported unit vectors are dense, choose successively integers
$1\le M_1\le N_1$ and $N_j+1<M_{j+1}\le N_{j+1}$, $j\ge1$, together with
unit vectors $\varphi_j$ such that $\varphi_j(n)=0$ unless
$M_j\le n\le N_j$ and
$\langle J_0\varphi_j,\varphi_j\rangle>
\beta_0-\varepsilon_j/2$.

Set $V=\operatorname{diag}(v_n)$, where $v_n=\varepsilon_j$ on
$\supp(\varphi_j)$ and $v_n=0$ otherwise, and let $J=J_0+V$.  Since
$v_n\rightarrow0$, \cite[Lemma~3.9 and its proof, p.~64]{Tes00} gives
$\sigma_{\rm ess}(J)=K_0$.  The off-diagonal coefficients of $J$ remain equal
to one.  Since $J_0$ is tridiagonal, $(J_0\varphi_j)(n)=0$ unless
$M_j-1\le n\le N_j+1$.  Thus $N_j+1<M_{j+1}$ gives
$\langle J_0\varphi_j,\varphi_i\rangle=0$ for $i\ne j$.  Moreover,
$V\varphi_j=\varepsilon_j\varphi_j$, and the vectors $\varphi_j$ are
orthonormal.  Hence, for $0\ne u=\sum_{j=1}^Nc_j\varphi_j$,
\begin{align*}
    \langle Ju,u\rangle
    &=\sum_{i,j=1}^N
    \langle (J_0+V)c_j\varphi_j,c_i\varphi_i\rangle\\
    &=\sum_{j=1}^N|c_j|^2
    \bigl(\langle J_0\varphi_j,\varphi_j\rangle+\varepsilon_j\bigr)\\
    &>\sum_{j=1}^N|c_j|^2
    \left(\beta_0+\frac{\varepsilon_j}{2}\right)\\
    &\ge\left(\beta_0+\frac{\varepsilon_N}{2}\right)
    \sum_{j=1}^N|c_j|^2\\
    &=\left(\beta_0+\frac{\varepsilon_N}{2}\right)\|u\|^2,
\end{align*}
where the last inequality uses $\varepsilon_j\ge\varepsilon_N$ for
$j\le N$.
For a Borel set $A$, let $P_A(J)$ denote the corresponding spectral
projection.  See \cite[p.~28]{Tes00}.  Lemma~4.6(i) of
\cite[p.~80]{Tes00} gives the inequality below.  Since
$\beta_0+\varepsilon_N/2>\max\sigma_{\rm ess}(J)$, the proof of
\cite[Lemma~9.28(b), pp.~281--282]{Luk22}, applied to $-J$, gives the
equality:
\begin{equation*}
    N\le \dim\operatorname{Ran}
    P_{(\beta_0+\varepsilon_N/2,\infty)}(J)
    =\sum_{\substack{x\in\sigma(J)\\
    x>\beta_0+\varepsilon_N/2}}\dim\ker(J-x).
\end{equation*}
The point spectrum of a half-line Jacobi matrix is simple
\cite[Remark~1.10, p.~18]{Tes00}.  Thus $J$ has at least $N$ distinct
eigenvalues above $\beta_0+\varepsilon_N/2$.  Writing the eigenvalues above
$\beta_0$ as $x_1>x_2>\cdots>\beta_0$, we have
$x_N-\beta_0>1/(2N^2)$ and $x_N\rightarrow\beta_0$.

Let $d\nu$ be the spectral measure of $J$ for $\delta_1$, put
$K=\supp(d\nu)$, and set $X=K\bs K_0$.
Proposition~\ref{zero-measure-schrodinger-prop} gives
$W_{2,n}(K,\nu)=1$.  The set $X$ consists of countably many isolated
eigenvalues and is therefore polar.  Hence $K_0$ is the regular part of $K$
and $g_K=g_{K_0}$.  Since $K_0\subset[\min K_0,\beta_0]$, monotonicity of
Green functions
\cite[Corollary~4.4.5]{Ran95} and \eqref{interval-endpoint-green} give
$g_{K_0}(\beta_0+t)\ge c\sqrt t$ for all sufficiently small $t>0$.
Therefore,
for some $N_0$,
\begin{equation*}
    \BC(K)\ge\sum_{N\ge N_0}g_{K_0}(x_N)
    \ge\frac{c}{\sqrt2}\sum_{N\ge N_0}\frac1N=\infty.
\end{equation*}
Thus $(B)$ fails.  Since $W_{2,n}(K,\nu)=1$, both $(U_2)$ and $(L_2)$ hold,
and Theorem~\ref{szego-implications-thm} shows that $(S)$ fails.  Finally,
$(P)$ fails because $K_0$ is not Parreau--Widom.  Hence
$(P,B,S,U_2,L_2)=(-,-,-,+,+)$.
\end{enumerate}
\end{example}

\subsection{The separation table}

The examples above give the following seventeen distinct patterns.
\begin{equation}\label{split-pattern-table}
\begin{array}{c|c|ccccc}
\text{example} & \text{range of }p &(P)&(B)&(S)&(U_p)&(L_p)\\
\hline
\ref{elementary-patterns-ex}(a)&0<p\le\infty&+&+&+&+&+\\
\ref{elementary-patterns-ex}(b)&0<p\le\infty&+&+&-&+&-\\
\ref{elementary-patterns-ex}(c)&0<p\le\infty&+&-&+&-&+\\
\ref{elementary-patterns-ex}(d)&2\le p\le\infty&-&+&+&-&+\\
\ref{elementary-patterns-ex}(e)&0<p\le\infty&-&-&+&-&+\\
\ref{discrete-patterns-cor}(a)&1<p<\infty&+&-&-&+&-\\
\ref{discrete-patterns-cor}(b)&1<p<\infty&-&+&-&+&-\\
\ref{discrete-patterns-cor}(c)&1<p<\infty&-&-&-&+&-\\
\ref{discrete-patterns-cor}(d)&1<p<\infty&+&-&-&-&-\\
\ref{discrete-patterns-cor}(e)&2\le p<\infty&-&+&-&-&-\\
\ref{discrete-patterns-cor}(f)&1<p<\infty&-&-&-&-&-\\
\ref{ex-B-L2-only}&p=2&-&+&-&-&+\\
\ref{ex-only-L2}&p=2&-&-&-&-&+\\
\ref{jacobi-patterns-ex}(a)&p=2&+&-&-&+&+\\
\ref{jacobi-patterns-ex}(b)&p=2&+&-&-&-&+\\
\ref{zero-measure-fixed-ex}(a)&p=2&-&+&-&+&+\\
\ref{zero-measure-fixed-ex}(b)&p=2&-&-&-&+&+
\end{array}
\end{equation}

Write the signs in the order $(P,B,S,U_p,L_p)$, and let $\ast$ denote
either sign.  Counting each pattern only at the first implication it
violates, the implications $(S)\Rightarrow(L_p)$,
$(S)+(U_p)\Rightarrow(B)$, $(P)+(B)\Rightarrow(U_p)$, and
$(P)+(B)+(L_p)\Rightarrow(S)$ exclude, respectively,
$(\ast,\ast,+,\ast,-)$ (eight patterns),
$(\ast,-,+,+,+)$ (two patterns),
$(+,+,-,-,\ast)$ together with $(+,+,+,-,+)$ (three patterns), and
$(+,+,-,+,+)$ (one pattern).  Thus these implications exclude fourteen
of the thirty-two formal sign patterns, subject, when $p=\infty$, to the
additional hypotheses on the weight in
Theorem~\ref{szego-implications-thm}.  For $2\le p\le\infty$,
Theorem~\ref{unconditional-pw-thm} also excludes
$(-,+,+,+,+)$, leaving seventeen possible patterns.  When $p=2$, all
seventeen occur in \eqref{split-pattern-table}; hence no additional
implication among $(P)$, $(B)$, $(S)$, $(U_2)$, and $(L_2)$ can hold in
general.  For other values of $p$, the table records the ranges covered
by the constructions but is not asserted to be exhaustive.

\bigskip
\noindent\textbf{Acknowledgements.} The authors gratefully acknowledge the American Institute of Mathematics for its support and hospitality during an AIM SQuaRE meeting, which led to the present work.

\end{document}